\documentclass{amsart}
\usepackage{marktext}
\usepackage{gimac}

\newcommand{\GI}[2][]{\sidenote[colback=yellow!20]{\textbf{GI\xspace #1:} #2}}
\newcommand{\SJ}[2][]{\sidenote[colback=green!10]{\textbf{SJ\xspace #1:} #2}}
\newcommand{\BC}[2][]{\sidenote[colback=orange!20]{\textbf{BC\xspace #1}: #2}}

\newcommand{\tmix}{t_\mathrm{mix}}

\DeclareMathOperator{\dist}{dist}

\DeclareMathOperator{\Leb}{Leb}
\DeclareMathOperator{\rank}{rank}

\DeclarePairedDelimiter{\ip}{\langle}{\rangle}
\newtheorem{question}[theorem]{Question}

\newcommand{\vertiii}[1]{{\left\vert\kern-0.25ex\left\vert\kern-0.25ex\left\vert #1
    \right\vert\kern-0.25ex\right\vert\kern-0.25ex\right\vert}}

\renewcommand{\setminus}{-}
\newtheorem{assumption}[theorem]{Assumption}

\begin{document}
\title[A Harris Theorem for Enhanced Dissipation]{A Harris Theorem for Enhanced Dissipation, and an Example of Pierrehumbert}

\author[Cooperman]{William Cooperman}
\address{%
  Courant Institute of Mathematical Sciences,
  New York University,  NY 10003
}
\email{bill@cprmn.org}

\author[Iyer]{Gautam Iyer}
\address{%
  Department of Mathematical Sciences, Carnegie Mellon University, Pittsburgh, PA 15213.
}
\email{gautam@math.cmu.edu}

\author[Son]{Seungjae Son}
\address{%
  Department of Mathematical Sciences, Carnegie Mellon University, Pittsburgh, PA 15213.
}
\email{seungjas@andrew.cmu.edu}
\begin{abstract}
  In many situations, the combined effect of advection and diffusion greatly increases the rate of convergence to equilibrium -- a phenomenon known as \emph{enhanced dissipation}.
  Here we study the situation where the advecting velocity field generates a random dynamical system satisfying certain \emph{Harris conditions}.
  If~$\kappa$ denotes the strength of the diffusion, then we show that with probability at least $1 - o(\kappa^N)$ enhanced dissipation occurs on time scales of order~$\abs{\ln \kappa}$, a bound which is known to be optimal.
  Moreover, on long time scales, we show that the rate of convergence to equilibrium is almost surely~\emph{independent} of diffusivity.
  As a consequence we obtain enhanced dissipation for the randomly shifted alternating shears introduced by Pierrehumbert '94.
\end{abstract}

\thanks{This work has been partially supported by the National Science Foundation under grants
  2108080, 
  2342349, 
  2303355, 
  2406853 
  and the Center for Nonlinear Analysis.}

\maketitle

\section{Introduction}\label{s:intro}

\subsection{Main Results}

We begin by stating our results.
Following this, we will survey the literature and place our work in the context of existing results.
Let~$u$ be a (possibly time dependent) divergence free vector on the torus, $\kappa > 0$, and~$\rho$ solve the advection diffusion equation
\begin{equation}\label{e:ad}
  \partial_t \rho + u \cdot \grad \rho - \kappa \lap \rho = 0 \,,
\end{equation}
on the $d$-dimensional torus~$\T^d$.
Multiplying~\eqref{e:ad} by~$\rho$, using the fact that~$\dv u = 0$, integrating and using the Poincar\'e inequality implies
\begin{equation}\label{e:EE-Poincare}
  \norm{\rho(\cdot, t) - \bar \rho}_{L^2} \leq e^{-\lambda_1 \kappa t} \norm{\rho_0 - \bar \rho }_{L^2} \,,
\end{equation}
where
\begin{equation}
  \bar \rho
    = \int_{\T^d} \rho_0(x) \, dx
  \,,
\end{equation}
is the (constant) equilibrium solution to~\eqref{e:ad} and $\lambda_1 > 0$ is the smallest non-zero eigenvalue of the negative Laplacian.
\emph{Enhanced dissipation} is the phenomenon where solutions to~\eqref{e:ad} converge to equilibrium faster than the upper bound~\eqref{e:EE-Poincare}.
Our main result shows that if the flow of~$u$ generates a random dynamical system (RDS) that satisfies the Harris conditions (stated below), then enhanced dissipation occurs.
The enhanced dissipation rate is optimal at short times with large probability, and is almost surely \emph{independent} of the diffusivity for long times.

\begin{theorem}\label{t:EDIntro}
  Suppose the flow of~$u$ generates a random dynamical system that satisfies the Harris conditions stated in Assumptions~\ref{a:fts}--\ref{a:flow}, below.
  For any~$\alpha > 0$, $q < \infty$ there exists $\gamma > 0$ and a $\kappa$-dependent random variable $D_\kappa$ such that for every initial data~$\rho_0 \in L^1(\T^d)$, the solution to~\eqref{e:ad} satisfies
  \begin{equation}\label{e:rho-L1-Linf}
    \norm{\rho(\cdot, t) - \bar \rho}_{L^\infty} \leq \frac{D_\kappa}{\kappa^{\frac{d}{2} + \alpha}}
      e^{- \gamma t}
      \norm{\rho_0 - \bar \rho }_{L^1} \,.
  \end{equation}
  Moreover, there exist a $\kappa$-independent, deterministic, constant $\bar D_q$ such that
  \begin{equation}\label{e:D-kappa-bound}
    \E D_\kappa^q \leq \bar D_q\,.
  \end{equation}
\end{theorem}
\begin{remark}\label{r:chebychev}
  The bound~\eqref{e:D-kappa-bound} implies that for any~$\beta > 0$, $\P( D_\kappa \geq \kappa^{-\beta} ) \leq \kappa^{\beta q} \bar D_q$.
  Using this in~\eqref{e:rho-L1-Linf} will show that with probability at least $1 - \bar D_q \kappa^{\beta q}$, we have
  \begin{equation}\label{e:ED-L1-Linf}
    \norm{\rho(\cdot, t) - \bar \rho}_{L^\infty} \leq \frac{1}{\kappa^{\frac{d}{2} + \alpha + \beta }}
      e^{- \gamma t}
      \norm{\rho_0 - \bar \rho }_{L^1}
    \,,
  \end{equation}
  for all $t \geq 0$, and all~$\rho_0 \in L^1$.
\end{remark}

An equivalent probabilistic formulation of this result is as follows.
Consider the Markov process defined by the SDE
\begin{equation}\label{e:SDEX}
  dX^{\kappa}_{t}(x) = - u(X^{\kappa}_{t}(x), t) \, dt + \sqrt{2\kappa} \, dW_t\,,
  \quad
  X^{\kappa}_{0}(x) = x \,,
\end{equation}
on the torus~$\T^d$.
Since $\dv u = 0$, the (unique) stationary distribution of~$X^\kappa$ is the Lebesgue measure.
Let $p^\kappa_t(x, y)$  denote the transition density of~$X^\kappa_t$, and recall the~\emph{uniform mixing time}~\cite{LevinPeres17,MontenegroTetali06} is defined by
\begin{equation}
  \tmix^\infty( X^\kappa, \epsilon )
    = \inf\set[\Big]{ t \geq 0 \st \sup_{x \in \T^d} \norm{p^\kappa_t(x, \cdot) - 1}_{L^\infty} < \epsilon }
  \,.
\end{equation}
It is easy to see that the~\emph{uniform mixing time} of~$X^\kappa$  satisfies
\begin{equation}
  \tmix^\infty(X^\kappa, \epsilon)
    \leq \frac{C \abs{\ln \epsilon}}{\kappa}
    \,,
\end{equation}
for some constant $C > 0$.
Theorem~\ref{t:EDIntro} is equivalent to the following uniform mixing time estimate.

\begin{theorem}\label{t:umix}
  Suppose the flow of~$u$ generates a random dynamical system satisfying the Harris conditions (Assumptions~\ref{a:fts}--\ref{a:flow}, below).
  For any~$\alpha > 0$, $q < \infty$ there exists $\gamma > 0$ and a $\kappa$-dependent random variable $D_\kappa$ such that
  \begin{equation}\label{e:tmix-bound}
    \tmix^\infty( X^\kappa, \epsilon )
      \leq \frac{1}{\gamma}
	\ln \paren[\Big]{ \frac{D_\kappa }{\epsilon \kappa^{\frac{d}{2} + \alpha}} }
  \end{equation}
  almost surely.
  Moreover, there exists a $\kappa$-independent (deterministic) constant $\bar D_q$ such that~\eqref{e:D-kappa-bound} holds.
\end{theorem}
\begin{remark}
  We clarify that~$\tmix^\infty(X^\kappa, \epsilon)$ is random as it depends on~$u$ (it is, of course, independent of~$W$).
  For any fixed~$\epsilon > 0$ (independent of~$\kappa$), Chebychev's inequality~\eqref{e:D-kappa-bound} and~\eqref{e:tmix-bound} show that for any~$\beta > 0$ we have
  \begin{equation}
    \P\paren[\Big]{
      \tmix^\infty(X^\kappa, \epsilon)
	\leq \frac{1}{\gamma} \ln \paren[\Big]{
	  \frac{1}{\epsilon \kappa^{\frac{d}{2} + \alpha + \beta} }
	}
      }
      \geq 1 - \bar D_q \kappa^{\beta q}
    \,.
  \end{equation}
  Moreover, as~$\epsilon \to 0$, we obtain the $\kappa$-independent uniform mixing time bound
  \begin{equation}
    \lim_{\epsilon \to 0} \frac{\tmix(X^\kappa, \epsilon)}{\abs{\ln \epsilon}}
      \leq \frac{1}{\gamma}
      \,,
    \quad\text{almost surely}
    \,.
  \end{equation}
\end{remark}

\begin{remark}\label{r:checkable-conditions}
  While Assumptions~\ref{a:fts}--\ref{a:flow} are easy to state (see Section~\ref{s:harris}, below), they aren't easy to verify in practice.
  Recent papers~\cite{BedrossianBlumenthalEA22,BlumenthalCotiZelatiEA22} instead assume certain conditions which are stronger than Harris conditions, but are easier to verify.
  For convenience of the reader, we state these conditions in Section~\ref{s:checkableConditions}, below.
\end{remark}

\begin{remark}[Pulsed diffusions]
  We can also obtain similar results for pulsed diffusions.
  Namely, define the Markov process $Y^\kappa$ by
  \begin{equation}
    Y^\kappa_{n+1} = \varphi_{n+1}(Y^\kappa_n) + \zeta^\kappa_{n+1}
    \,,
  \end{equation}
  where~$\zeta^\kappa_n$ are i.i.d.\ periodized Gaussians with variance~$\kappa$, and~$\varphi_n$ is a RDS on~$\T^d$, independent of~$\zeta^\kappa$.
  If the RDS~$\varphi_n$ satisfies the Harris conditions (Assumption~\ref{a:fts}--\ref{a:submersion-2point}), then we obtain the same mixing time bound~\eqref{e:tmix-bound} for the process~$Y^\kappa$.
  The proof is similar to the proof of Theorem~\ref{t:umix}, and in this setting several technical steps become much simpler.
\end{remark}

As an immediate consequence, we can show enhanced dissipation if~$u$ is obtained by randomly shifting and alternating sinusoidal shears.

\begin{corollary}\label{c:alt-shear}
  Let $A > 0$, $d = 2$ and $\zeta_n$ a sequence of i.i.d.\ random variables that are uniformly distributed on $[0, 1]$.
  For $n \in \N$ and $t \in [2n, 2n+2)$ define
  \begin{equation}\label{e:alt-shear-def}
    u(x, t) =
      \begin{cases}
	A \sin( 2\pi (x_2 - \zeta_{2n}) ) \, e_1
	  & t \in [2n, 2n+1) \,,
	\\
	A \sin( 2\pi (x_1 - \zeta_{2n+1}) ) \, e_2
	  & t \in [2n+1, 2n+2) \,,
      \end{cases}
  \end{equation}
  where $x = (x_1, x_2) \in \T^2$, and $e_i$ is the $i$-th standard basis vector.
  Then, almost surely, we have the enhanced dissipation bound~\eqref{e:rho-L1-Linf} for a random variable~$D_\kappa$ that satisfies the $\kappa$-independent bounds~\eqref{e:D-kappa-bound}.
  Consequently, with probability at least $1 - \bar D_q \kappa^{\beta q}$, the enhanced dissipation bound~\eqref{e:ED-L1-Linf} holds.
\end{corollary}

This example was introduced by Pierrehumbert~\cite{Pierrehumbert94} and~\cite{BlumenthalCotiZelatiEA22} recently verified that it satisfies conditions that are stronger than our Harris assumptions.
As a result the Corollary~\ref{c:alt-shear} follows immediately from Theorem~\ref{t:EDIntro} and~\cite{BlumenthalCotiZelatiEA22} (see Proposition~\ref{p:checkable-conditions}, below).
We mention, however, that even though~\cite{BlumenthalCotiZelatiEA22} abstractly check Assumption~\ref{a:Lyapunov} (the existence of a Lyapunov function) for~\eqref{e:alt-shear-def}, the system is simple enough that a Lyapunov function can be constructed explicitly.
We do this in Proposition~\ref{p:alt-shear-lyapunov} in Section~\ref{s:alt-shear-lyapunov}, below.

\begin{remark}
  Instead of using a sine shear profile in~\eqref{e:alt-shear-def}, we can use shears with a piecewise linear profile.
  In this case the results in~\cite{ChristieFengEA23} will show that Assumption~\ref{a:fts}--\ref{a:flow} are satisfied, and so Theorems~\ref{t:EDIntro}, \ref{t:umix} apply and will give the enhanced dissipation bound~\eqref{e:rho-L1-Linf}, and the equivalent uniform mixing time bound~\eqref{e:tmix-bound}.
\end{remark}

\subsection{Motivation and Literature Review}

We now survey the literature and place our results in the context of existing results.
\emph{Enhanced dissipation} is a phenomenon that can be observed in every day life:
Pour some cream in your coffee.
If left alone, it will take hours to mix.
Stir it a little and it mixes right away.
This effect arises due to the interaction between the advection (stirring) and diffusion, and plays an important role in many applications concerning hydrodynamic stability and turbulence and occurs on scales ranging from micro fluids to meteorological / cosmological~\cite{
  LinThiffeaultEA11,
  Thiffeault12,
  Aref84,
  StoneStroockEA04
}.

To describe this mathematically, let $u$ be the velocity field of the ambient incompressible fluid, and~$\rho$ denote the concentration of a passively advected solute with molecular diffusivity $\kappa > 0$.
The evolution of~$\rho$ is governed by the advection diffusion equation~\eqref{e:ad}.
For simplicity, in this paper we only consider~\eqref{e:ad} with periodic boundary conditions on the $d$-dimensional torus $\T^d$.

If the ambient fluid is incompressible, the velocity field~$u$ satisfies divergence free condition
\begin{equation}\label{e:incompressible}
  \dv u = 0\,.
\end{equation}
In this case, an elementary energy estimate shows~\eqref{e:EE-Poincare} and hence the $L^2$ distance of the concentration from the equilibrium distribution decreases at most exponentially with a rate proportional to~$\kappa$.

Of course, \eqref{e:EE-Poincare} is only a crude upper bound.
In many practical situations one expects the convergence to happen much faster than~\eqref{e:EE-Poincare}.
Indeed, the advection term $u \cdot \grad \rho$ typically causes filamentation and moves energy towards small scales.
The diffusion term $\kappa \lap \rho$ damps small scales faster, and the combination of these two effects leads to enhanced dissipation -- faster convergence of~$\rho(\cdot, t)$ to~$\bar \rho$.

Several authors have proved enhanced dissipation by showing all solutions to~\eqref{e:ad} satisfy the decay estimate
\begin{subequations}
  \begin{equation}\label{e:EDL2Decay}
    \norm{\rho(\cdot, t) - \bar \rho}_{L^2} \leq \exp\paren[\Big]{ -\paren[\Big]{ \frac{t}{T(\kappa)}  - 1 }^+ }
  \norm{\rho_0 - \bar \rho }_{L^2} \,,
  \end{equation}
  for every $t \geq 0$, and some time scale $T(\kappa)$ for which
  \begin{equation}\label{e:Tkappa-small}\noeqref{e:Tkappa-small}
    \lim_{\kappa \to 0} \kappa T(\kappa) = 0\,.
  \end{equation}
\end{subequations}
Seminal work of Constantin et\ al.~\cite{ConstantinKiselevEA08} (see also~\cite{Zlatos10,KiselevShterenbergEA08}) shows that if $u$ is time independent, then such a $T(\kappa)$ exists if and only if $u \cdot \grad$ has no eigenfunctions in $H^1$.
For shear flows classical work of Kelvin~\cite{Kelvin87} shows one can choose $T(\kappa) = \kappa^{-\alpha}$ for some $\alpha < 1$.
There are now several results studying enhanced dissipation in more generality and for nonlinear equations (see~\cite{FannjiangNonnenmacherEA04,Wei19,BedrossianCotiZelati17,CotiZelatiDrivas21,FengMazzucatoEA23,CotiZelatiGallay23,AlbrittonBeekieEA22,CobleHe23,Seis23}).

The purpose of this paper is to further investigate the link between enhanced dissipation and mixing properties of~$u$.
Recall, a velocity field~$u$ is said to be \emph{exponentially mixing} if, in the absence of diffusion, a dye that is initially localized to ball of size~$\epsilon$ will get spread throughout the torus in time $O(\abs{\ln \epsilon})$ (see for instance~\cite{SturmanOttinoEA06}).
In the presence of diffusion, a dye localized to a point gets spread to a ball of size $O(\sqrt{\kappa})$ in time $O(1)$.
If $u$ is exponentially mixing, then this dye is spread throughout the torus by the flow in time $O(\abs{\ln \kappa})$.
As a result, in this case we expect~\eqref{e:EDL2Decay} should hold with $T(\kappa) = O(\abs{\ln \kappa})$.

Surprisingly, this is not easy to prove, and is an open question in this generality.
Currently available results~\cite{FengIyer19,Feng19,CotiZelatiDelgadinoEA20} show that if~$u$ is exponentially mixing, then one can choose~$T(\kappa) = O(\abs{\ln \kappa}^2)$ in~\eqref{e:EDL2Decay}.
For a few specific exponentially mixing systems, available results~\cite{BedrossianBlumenthalEA21,ElgindiLissEA23,IyerLuEA23} show that one can choose~$T(\kappa) = O(\abs{\ln \kappa})$ in~\eqref{e:EDL2Decay}.
However, to the best of our knowledge, there is no general theorem (in discrete or continuous time) that shows that for any exponentially mixing flow one can choose~$T(\kappa) = O(\abs{\ln \kappa})$ in~\eqref{e:EDL2Decay}.

One elementary observation is that if almost every realization of the stochastic flows~$X^\kappa$ is exponentially mixing, then one has enhanced dissipation as in~\eqref{e:EDL2Decay} with~$T(\kappa) = C \abs{\ln \kappa}$ for some~$C$ that can be explicitly computed in terms of the mixing rate.
Thus, a natural question to ask is is whether or not the notion of exponentially mixing is stable with respect to~$\kappa$.
\begin{question}\label{q:XKappaExpMix}
  If~$u$ is exponentially mixing, then for sufficiently small~$\kappa > 0$ must almost every realization of~$X^\kappa$ be exponentially mixing (with a controlled rate)?
\end{question}

Since the notion of exponentially mixing involves the long time behavior, it is not easy to determine the answer to Question~\ref{q:XKappaExpMix}.
We instead look for stronger conditions on~$u$ which will will guarantee that for all sufficiently small~$\kappa > 0$, almost every realization~$X^\kappa$ is exponentially mixing (with a controlled rate).

A general principle that is well known to the Sinai school is that for \emph{random dynamical systems (RDS)}, geometric ergodicity of the two point process implies almost sure exponential mixing (see~\cite{DolgopyatKaloshinEA04,BedrossianBlumenthalEA22}, or the proof of Lemma~\ref{l:X-exp-mix}, below).
One could then ask whether or not this property is stable in~$\kappa$.

\begin{question}\label{q:twopoint}
  If the two point process associated to the flow of a random dynamical system~$u$ is geometrically ergodic, then is the two point processes associated to the SDE~\eqref{e:SDEX} also geometrically ergodic for all small~$\kappa > 0$?
\end{question}

If Question~\ref{q:twopoint} is answered affermatively, then for all sufficiently small~$\kappa > 0$, must almost every realization~$X^\kappa$ be exponentially mixing.
Not surprisingly, this question is also hard to answer.
Geometric ergodicity involves questions about long time limits which are not stable as~$\kappa$ varies.
There is, however, a classical result of Harris~\cite{Harris55,MeynTweedie09} that proves geometric ergodicity of a Markov process provided there is a Lyapunov function, and a small set.
A version of this condition turns out to be stable in~$\kappa$ (see Lemma~\ref{l:harrisassumptions}, below) which in turn leads to almost sure exponential mixing of the flows~$X^\kappa$ (see Lemma~\ref{l:X-exp-mix}, below), which in turn yields Theorem~\ref{t:EDIntro}.


\GI[2024-03-16]{Fix references and edit}
Finally, we mention that if $u$ is regular, and uniformly bounded in time, then we must have $T(\kappa) \geq O(\kappa)$ in~\eqref{e:EDL2Decay} (see for instance~\cite{MilesDoering18,Poon96,BedrossianBlumenthalEA21,Seis22}).
When $u$ is irregular, one can even choose $T(\kappa) = O(1)$.
This is known as~$\emph{anomalous dissipation}$, and examples of this were recently proved in~\cite{DrivasElgindiEA22,ColomboCrippaEA22,ArmstrongVicol23}.
In this paper we only consider regular velocity fields, and so are in a situation where anomalous dissipation can not ocur.

\subsection*{Plan of this paper.}

In Section~\ref{s:harris} we define our notation and state Assumption~\ref{a:fts}--\ref{a:flow} used in Theorems~\ref{t:EDIntro} and~\ref{t:umix}.
In Section~\ref{s:main-proof} we state the three lemmas (Lemmas~\ref{l:harrisassumptions}--\ref{l:X-exp-mix}) that will quickly yield Theorems~\ref{t:EDIntro} and~\ref{t:umix}, and use these lemmas to prove Theorems~\ref{t:EDIntro} and~\ref{t:umix}.
The first of these lemmas (Lemma~\ref{l:harrisassumptions}) is the main new contribution of this paper and which guarantees that our assumptions imply the Harris conditions hold for all sufficiently small~$\kappa > 0$.
The proof of Lemma~\ref{l:harrisassumptions} is split up into two steps -- the existence of a Lyapunov function in Section~\ref{s:Lyapunov}, and the existence of a $\kappa$-independent small set in Section~\ref{s:small-set}.
Following this, we prove Lemma~\ref{l:harrisassumptions} in Section~\ref{s:verify-harris-assumptions}.
The proof of Lemma~\ref{l: uge} uses a quantitative version of Harris's theorem~\cite{HairerMattingly11} and is presented in Section~\ref{s:uge}.
The proof of Lemma~\ref{l:X-exp-mix} is based on the idea in~\cite{DolgopyatKaloshinEA04} and is presented in Section~\ref{s:bc}.
Finally, we conclude this paper by explicitly finding a Lyapunov function for the randomly shifted alternating shear example in~\ref{c:alt-shear}.

\subsection*{Acknowledgements.}

The authors wish to thank Alex Blumenthal and Sam Punshon-Smith for helpful comments.

\section{Notation and Preliminaries.}\label{s:assumptions}

In this section we set up our notational convention and state the assumptions required for Theorem~\ref{t:EDIntro} and~\ref{t:umix}.

\subsection{Notation and setup.}
We will now define a setup that chooses the velocity field~$u$ randomly from a finite dimensional family of~$C^2$, incompressible vector fields, and repeats this choice on after time intervals of length~$1$.
Let $\mathscr M$ be a complete smooth Riemannian manifold, and $\mathscr U \colon \mathscr M \times \T^d \times [0, 1] \to \R^d$ be a $C^2$-function such that
\begin{equation}
  \nabla_x \cdot \mathscr U(\xi, x, t) = 0
  \,.
\end{equation}

Now let $(\Omega_0, \mathcal{F}_0, \P_0)$ be a probability space and let $\omega = (\omega_0, \omega_1, \dots )$ be a sequence of $\mathscr M$-valued random variables whose distribution is absolutely continuous with respect to the volume measure on~$\mathscr M$.
Define the (random) velocity field~$u$ by
\begin{equation}
  u(x, t) \defeq \mathscr U(\omega_n, x, t-n) \quad\text{when } t \in [n, n+1 )
  \,,
\end{equation}
where, following standard convention, we suppress the dependence of~$u$ on~$\omega$.
The flow of~$u$ is defined by the ODE
\begin{equation}\label{e:RDSu}
  \partial_t X^0_t = u( X^0_t, t )
  \,,
  \qquad
  X^0_0 = \mathrm{Id}
  \,.
\end{equation}
Restricting~$X^0$ to integer times gives a Markov process, which we refer to as the RDS generated by~$u$.

To define the processes~$X^\kappa$, let $(\Omega_W, \mathcal{F}_W, \P_W)$ be a probability space, and $W$ be a $\T^d$-valued Brownian motion on this space.
We will now consider the product space $\Omega = \Omega_W \times \Omega_0$ with the product $\sigma$-algebra $\mathcal F = \mathcal F_W \otimes \mathcal F_0$ and product measure $\P \defeq \P_W \otimes \P_0$.
By a slight abuse of notation, we will sometimes treat random variables on each of the coordinate spaces~$\Omega_0$, $\Omega_W$ as random variables on the product space~$\Omega$ by composing with the corresponding coordinate projection.
In this sense, we may treat~$u, W$ as independent processes on~$\Omega$, and let $X^\kappa_{t}$ be the solution of~\eqref{e:SDEX} on~$\T^d$.

\subsection{Two point and projective chains.}
In order to state the Harris conditions, we need to define the~\emph{two point} and~\emph{projective chains}, and formalize the dependence on the noise history.
Given $\xi = (\xi_0, \dots, \xi_{n-1}) \in \mathscr M^n$, define $\mathscr U_n \colon \mathscr M^n \times \T^d \times [0, n) \to \R^d$ by
\begin{equation}
  \mathscr U_n( \xi, x, t ) = \mathscr U( \xi_m, x, t-m )
    \quad\text{if } t \in [m, m + 1)
  \,.
\end{equation}
Now define $\mathscr X_n\colon \mathscr M^n \times \T^d \to \T^d$ to be the flow of $\mathscr U_n$ after time~$n$.
That is, define
\begin{equation}
  \mathscr X_n(\xi, x) = \varPhi_{n}(x)\,,
\end{equation}
where~$\varPhi$ is defined by
\begin{equation}\label{e:flow}
  \partial_t \varPhi_{t} = \mathscr U_n( \xi, \varPhi_{t}, t )\,,
  \quad
  \varPhi_{0} = \mathrm{Id} \,.
\end{equation}

Given $\xi \in \mathscr M^n$, $(x, v) \in T \T^d$ (the tangent bundle of the torus), and $y \in \T^d$ with $\norm{v} = 1$, $y\neq x$, we define the derivative, projective and two point maps by
\begin{equation}
  \mathscr A_n( \xi, x )
    \defeq D_x \mathscr X_n(\xi, x )\,,
  \quad
  \hat{\mathscr X}_n(\xi, x, v) =
    \paren[\Big]{\mathscr X_n(\xi, x),
      \frac{\mathscr A_n(\xi, x) v}{\norm{\mathscr A_n(\xi, x) v}}
    }
\end{equation}
and
\begin{equation}\label{e:curlyX2}
  \mathscr X_n^{(2)}(\xi, x, y) = (\mathscr X_n(\xi, x), \mathscr X_n(\xi, y))
  \,,
\end{equation}
respectively. Note that $\hat{\mathscr X_n}$ takes values on $S\T^d$, the unit sphere bundle (see for instance~\cite{doCarmo92}) defined by
\begin{equation}
S\T^d \defeq \set{(x,v)\in T\T^d\st \norm{v}=1}\,.
\end{equation}
The map $\mathscr X^{(2)}_n$ takes values on $\T^{d, (2)}$ where
\begin{gather}
\T^{d, (2)}\defeq \T^d \times \T^d \setminus \Delta\,,
\end{gather} and
$\Delta \defeq \set{(x,y) \in \T^d \times \T^d \st x=y}$.
\medskip

Notice that the RDS defined by~\eqref{e:RDSu} can be written in terms of~$\mathscr X_n$ using the identity
\begin{equation}
  X^0_n(x) = \mathscr X_n( \underline{\omega_n}, x )\,,
  \quad\text{where}\quad
  \underline{\omega_n} \defeq (\omega_0, \dots, \omega_{n-1}) \in \mathscr M^n
  \,,
\end{equation}
for~$n \geq 1$.
We clarify that for~$t \in [n-1, n)$ the velocity field~$u$ in~$\eqref{e:RDSu}$ is determined from~$\omega_{n-1}$,  the~$n^\text{th}$ term in~$\omega$.
The flow map~$X^0_n$, however, relies on the entire history~$\underline{\omega_n} = (\omega_0, \dots, \omega_{n-1})$.

We also define the derivative, projective and two point processes by
\begin{equation}
  A^0_n(x) = \mathscr A_n( \underline{\omega_n}, x )
  \,,
  \quad
  \hat X_n(x, v) = \hat{\mathscr X}_n( \underline{\omega_n}, x, v )\,,
  \quad
  X^{0, (2)}_n(x, y) = \mathscr X^{(2)}_n( \underline{\omega_n}, x, y )
  \,,
\end{equation}
where as before $(x, v) \in S \T^d$ and $(x,y) \in \T^{d, (2)}$.
We denote the $n$-step transition kernel of these processes by
\begin{gather}
  P_0^n(x,A) \defeq \P_0[X_n^0(x)\in A]
  \,,\\
  \quad
  \hat P_0^n ((x,v), A) \defeq \P_0\Big[\hat X_n^0(x,v)\in A\Big]\,,
\\
  \label{e:two-point-def}
  P_0^{(2), n}((x,y), A) \defeq \P_0\big[X_n^{0, (2)} \in A\big]\,.
\end{gather}
For brevity, when $n = 1$ we will drop the superscript.
The action of $P_0^n$ on functions is defined defined by
\begin{equation}
  P_0^n f(x)
  = \int_{\T^d} P_0^n( x, dy ) f(y) 
  = \E_0 f(X^0_n(x))
  \,,
\end{equation}
where~$\E_0$ is the expectation with respect to the measure~$P_0$.
The action of~$\hat P_0^n$ and~$P_0^{(2), n}$ on functions are defined similarly.

For~$\kappa > 0$, we define the Markov process $X^\kappa$ obtained by restricting the solution to~\eqref{e:SDEX} to integer times.
Given $n \in \N$, let $P^n_\kappa(x, \cdot)$ denote the transition probability of $X^\kappa_{n}$.
That is, for $n \in \N$ we define
\begin{equation}
  P^n_\kappa(x, A) \defeq \P( X^\kappa_{n}(x) \in A )
  \,.
\end{equation}
As before, when $n = 1$, we drop the superscript and simply write $P_\kappa$ for $P^1_\kappa$.
The two point process~$X^{\kappa,(2)}$ is  defined analogously by
\begin{equation}
  X^{\kappa, (2)}_n(x, y) = (X^\kappa_n(x), X^\kappa_n(y) )
  \,,
\end{equation}
and its $n$-step transition kernel will be denoted by~$P_\kappa^{(2), n}$.
\smallskip

\subsection{Harris Conditions.}\label{s:harris}

Harris theorems typically assume the existence of a Lyapunov function and a small set, and show that the associated system is uniformly geometrically ergodic~\cite{MeynTweedie09,Harris55,HairerMattingly11}.
We will now state these conditions for the two point chain~$P^{(2)}_0$ defined in~\eqref{e:two-point-def}.
We reiterate that our assumptions only concern the RDS generated by the flow of~$u$ with~$\kappa = 0$.

\begin{assumption}\label{a:fts}
The two point chain with kernel $P_0^{(2)}$ is Feller, topologically irreducible, strongly aperiodic.
\end{assumption}

Our next assumption concerns the existence of a Lyapunov function for the two point chain near a small neighborhood of the diagonal.
To state this, let $s > 0$ and define a punctured neighborhood of the diagonal~$\Delta(s)$ by
\begin{equation}
  \Delta(s)\defeq \{(x,y)\in \T^d \times \T^d  \st 0 < d(x,y) < s\}\,,
\end{equation}
where $d(x, y)$ is the torus distance between~$x$ and~$y$.
On the flat torus it is easy to see that for $x, y \in \Delta(1/2)$, we can find a unique $v \in T_x \T^d$ such that $\exp_x(v) = y$, where $\exp_x$ denotes the exponential map (see for instance~\cite{doCarmo92}).

\begin{assumption} \label{a:Lyapunov}
There exists a Lyapunov function~$V\colon  \T^{d, (2)} \to [1,\infty)$ and constants $\tilde \gamma \in (0,1)$, $s_* \in (0,1/2)$ such that
  \begin{equation}\label{e:LyapunovEpEq0}
    P^{(2)}_0 V < \tilde \gamma V
    \quad\text{on } \Delta(s_*)
    \,.
    \end{equation}
    Moreover, there exists $p \in (0, 1)$ such that the function $V$ is of the form
    \begin{equation}\label{e: LyapunovV}
      V(x,y) = d(x,y)^{-p}\psi\paren[\Big]{x, \frac{\exp_x^{-1}(y)}{d(x,y)}}
      \quad\text{on } \Delta(s_*)
      \,.
    \end{equation}
    where $\psi:S\T^d\to \R^+$ is a continuous, strictly positive function.
\end{assumption}

The standard Harris Theorem~\cite{MeynTweedie09} requires the assumptions~\ref{a:fts},~\ref{a:Lyapunov}, and the existence of a small set for $P_0^{(2)}$, and shows ergodicity of the chain.
However, for Lemma~\ref{l:X-exp-mix}, we will need to show that for sufficiently small~$\kappa > 0$ there exists a small set with~$\kappa$-independent bound on the minorizing measure.
We are presently unable to do this assuming only the existence of a small set for $P_0^{(2)}$.
We can, however, prove this under a slightly stronger condition which is not hard to verify in practice.

\begin{assumption}\label{a:submersion-2point}
There exist $n \geq 1$ and $(\xi_*, x_*)\in \mathscr M^n \times \mathscr \T^{d, (2)}$ such that the following hold.
\begin{enumerate}
  \item
    Let~$\rho_n$ be the density of the~$\mathscr M^n$-valued random variable $\underline{\omega_n}$.
    There exists $c, \epsilon >0$ such that for every $\xi \in \mathscr M^n$ with $\abs{\xi - \xi_*} < \epsilon$, we have $\rho_n(\xi) \geq c >0$.

  \item
    The map $\mathscr X^{(2)}_n(\cdot,x_*)\colon \mathscr M^n \to \mathscr \T^{d, (2)}$ is a submersion at $\xi=\xi_*\,.$
\end{enumerate}
\end{assumption}

Finally, we need a uniform bound on the velocity field so that the gradients of the diffeomorphisms~$\mathscr X_n$ are controlled uniformly in the noise.
\begin{assumption}\label{a:flow}
  The function~$\mathscr U\colon \mathscr M \times \T^d \times[0, 1] \to \R^d$ is
   such that
  \begin{align}
  \sup_{(\xi, t) \in \mathscr M \times [0,1]} \norm{\mathscr U(\xi, \cdot, t)}_{C^2(\T^d)} &< \infty\\
  \llap{\text{and}\quad}
  \sup_{(x,t)\in \T^d\times[0,1]}\norm{\nabla_\xi \mathscr U(\cdot, x,t)}_{L^\infty(\mathscr M)} &<\infty\,.
  \end{align}
\end{assumption}

We will prove (Section~\ref{s:Lyapunov}, below) that Assumptions~\ref{a:Lyapunov} and~\ref{a:flow} imply that~$V$  is a Lyapunov function for $P_\kappa^{(2)}$, and satisfies the drift condition~\eqref{e:LyapunovEpEq0} with slightly larger constants.
Following this, we will show (Lemma~\ref{l:harrisassumptions}, below) that these assumptions will also ensure $P_\kappa^{(2)}$ satisfies the assumptions of the Harris theorem with $\kappa$-independent constants.
As a result, a quantitative Harris theorem~\cite{HairerMattingly11} will show that $P_\kappa^{(2)}$ is uniformly geometrically ergodic with a $\kappa$-independent rate (Lemma~\ref{l: uge} in Section~\ref{s:uge}, below).
Once this has been established, an argument of~\cite{DolgopyatKaloshinEA04} will imply~$X^\kappa$ is exponentially mixing, and prove Lemma~\ref{l:X-exp-mix} (Section~\ref{s:bc}).

\subsection{Checkable Conditions that Guarantee the Harris Conditions.}\label{s:checkableConditions}

In practice it is not easy to find a Lyapunov function (Assumption~\ref{a:Lyapunov}).
As mentioned in Remark~\ref{r:checkable-conditions}, recent papers~\cite{BedrossianBlumenthalEA22,BlumenthalCotiZelatiEA22} instead assume certain conditions which are stronger than Harris conditions, but are easier to verify.
For convenience of the reader, we state these conditions here.


\begin{assumption}\label{a:checkableConditions1}
    The transition kernels $P_0$, $\hat P_0$ and $P^{(2)}_0$ are all Feller, and topologically irreducible.
\end{assumption}

\begin{assumption}\label{a:checkableConditions2}
Let $\mathscr N$ be one of the spaces $\T^d, S\T^d,$ and $\T^{d, (2)}$ and $\mathscr Y_k$ be one of the corresponding maps $\mathscr X_k$, $\hat {\mathscr X}_k$, or  $\mathscr X^{(2)}_k$.
For each choice of $\mathscr N$ and $\mathscr Y$, there exist $n \geq 1$ and $(\xi_*, x_*)\in \mathscr M^n \times \mathscr N$ such that the following hold.
\begin{enumerate}
  \item
    Let~$\rho_n$ be the density of the~$\mathscr M^n$-valued random variable $\underline{\omega_n}$.
    There exists $c, \epsilon >0$ such that for every $\xi \in \mathscr M^n$ with $\abs{\xi - \xi_*} < \epsilon$, we have $\rho_n(\xi) \geq c >0$.

  \item
    The map $\mathscr Y_n(\cdot,x_*)\colon \mathscr M^n \to \mathscr N$ is a submersion at $\xi=\xi_*\,.$
\end{enumerate}
\end{assumption}
\begin{assumption}\label{a:checkableConditions3}
  For each choice of $\mathscr N$ and $\mathscr Y$ as in Assumption~\ref{a:checkableConditions2}, there exist $\xi_{**} \in \supp (\dist(\zeta_1))$ and $y_* \in \mathscr N$ such that $\mathscr Y_1(\xi_{**}, y_*) = y_*$.
\end{assumption}
\begin{assumption}\label{a:checkableConditions4}
  Let $\mathscr N=\T^d, \mathscr Y=\mathscr X$, and let $n$, $(\xi_*, x_*)$ be as in Assumption~\ref{a:checkableConditions2}.
  Define a $C^1$-mapping $g\colon \mathscr M^n \to \text{SL}_d(\R)$ by
  \begin{equation}
    g(\xi) = \frac{1}{\abs{\det \mathscr A_n(\xi, x_*)}^\frac{1}{d}}\mathscr A_n(\xi, x_*)
    \,.
  \end{equation}
  Then, the restriction of the derivative $D_{\xi_*}g$ to $\ker D_\xi \mathscr X_n(\cdot, x_*) \subset T_\xi \mathscr M^n$ is surjective onto $T_{g(\xi_*)}\text{SL}_d(\R)$.
  \GI[2024-02-21]{Check if this is $\xi$ or $\xi_*$.}
\end{assumption}

These conditions are stronger than the Harris conditions in the following sense.
\begin{proposition}\label{p:checkable-conditions}\GI[2024-03-05]{Cleanup statement}
  If Assumption~\ref{a:flow} and~\ref{a:checkableConditions1}--\ref{a:checkableConditions4} hold, then the Assumptions~\ref{a:fts}--\ref{a:submersion-2point} also hold.
  Hence, in this case, the enhanced dissipation estimate~\eqref{e:rho-L1-Linf} also holds for some random variable~$D_\kappa$ (depending on~$u$ and~$\kappa$, but independent of~$W$) satisfying~\eqref{e:D-kappa-bound}.
\end{proposition}
\begin{proof}[Proof of Proposition~\ref{p:checkable-conditions}]
  The proof of Proposition~\ref{p:checkable-conditions} follows immediately from the results in~\cite{BlumenthalCotiZelatiEA22}.
  Assumption~\ref{a:checkableConditions1} along with continuity of~$\mathscr U$ (Assumption~\ref{a:flow}) and Assumption~\ref{a:checkableConditions3} implies the conditions in Assumption~\ref{a:fts}.
  Assumption~\ref{a:submersion-2point} follows immediately from Assumption~\ref{a:checkableConditions2}.
  To obtain the Lyapunov function (Assumption~\ref{a:Lyapunov}) we will apply Proposition~3.3 in~\cite{BlumenthalCotiZelatiEA22}.
  In order to do this we note that a standard Gronwall argument (see for instance Lemma~\ref{l:gronwall-difference}, below) and compactness imply that there exists a constant~$C_0'$ such that for all~$\xi \in \mathscr M$, $x, y \in \T^d$, we have
  \begin{equation}\label{e:LipschitzTorus}
    \frac{1}{C_0'}d(x,y)\leq d(\mathscr X_1(\xi, x), \mathscr X_1(\xi, y)) \leq C_0' d(x,y)
  \,.
  \end{equation}
  Along with Assumption~\ref{a:checkableConditions4}, this allows us to apply Proposition~3.3 in~\cite{BlumenthalCotiZelatiEA22}, and gives positivity of the top Lyapunov exponent.
  Now Proposition~4.5 in~\cite{BlumenthalCotiZelatiEA22} (and Assumptions~\ref{a:flow}--\ref{a:checkableConditions4})  will imply the existence of a Lyapunov function as stated in Assumption~\ref{a:Lyapunov}.
\end{proof}

\section{Proof of Theorem~\ref{t:EDIntro}}\label{s:main-proof}

As mentioned earlier, the main idea behind the proof of Theorems~\ref{t:EDIntro}, \ref{t:umix} is to show that Assumption~\ref{a:fts}--\ref{a:flow} imply that the Harris conditions hold for all sufficiently small~$\kappa > 0$.

\begin{lemma} \label{l:harrisassumptions}
Suppose Assumptions~\ref{a:fts}--\ref{a:flow} hold. Then there exist $l\in \N$, $\gamma_3\in (0,1)$, $K>0$, $R>\frac{2K}{1-\gamma_3}$, $\alpha \in (0,1)$, and a probability measure $\nu$, such that for all sufficiently small~$\kappa > 0$ we have
\begin{gather}
P_\kappa^{(2),l}V \leq \gamma_3 V + K\,, \label{e:Lyapunov-bound-multistep} \\
\label{e:V-sublevel-small}
\inf_{x\in \{V\leq R\}} P_\kappa^{(2),l}(x, \cdot) \geq \alpha\nu(\cdot)
\,.
\end{gather}
\end{lemma}

Once Lemma~\ref{l:harrisassumptions} is proved, the remainder of the proof can be obtained using established methods (see for instance~\cite{BlumenthalCotiZelatiEA22,BedrossianBlumenthalEA21,DolgopyatKaloshinEA04}).
First, a quantitative version of the Harris theorem~\cite{HairerMattingly11,MeynTweedie09} combined with Lemma~\ref{l:harrisassumptions} will show that the two point process~$X^{\kappa, (2)}$ is $V$-geometrically ergodic.

\begin{lemma}[V-geometric ergodicity] \label{l: uge}
  Suppose that~\eqref{e:Lyapunov-bound-multistep} and ~\eqref{e:V-sublevel-small} hold.
  There exist constants $C>0$ and $\beta > 0$ such that for all sufficiently small~$\kappa \geq 0$, all measurable $\varphi:\T^{d, (2)}\to \R$ such that $\|\varphi\|_V < \infty$, and any $n\in \N$, we have
  \begin{equation}
    \norm[\Big]{ P_\kappa^{(2),n} \varphi-\int \varphi \, d\pi^{(2)} }_V
      \leq Ce^{-\beta n} \norm[\Big]{ \varphi-\int \varphi \, d\pi^{(2)}}_V\,. \label{e: ugeGeneral}
  \end{equation}
\end{lemma}

Combining this with Borel-Cantelli argument in~\cite{KaloshinDolgopyatEA05} (see also~\cite{BedrossianBlumenthalEA21,BlumenthalCotiZelatiEA22}) will show almost sure exponential mixing of~$X^\kappa$.

\begin{lemma}\label{l:X-exp-mix}
  Suppose that~\eqref{e: ugeGeneral} holds.
  Then, for every~$\alpha > 0$ and $0<q<\infty$, there exists a random $D_\kappa \geq 1$ (which depends on~$u$ but is independent of~$W$), and deterministic~$\gamma > 0$ (independent of~$\kappa$) such that
  for all sufficiently small $\kappa \geq 0$,
  every pair of
  mean-zero test functions $f, g \in \dot H^{\alpha}$, and every $n \in \N$ we have
  \begin{equation}\label{e:X-exp-mix-disc-time}
     \ip{f, g\circ X^\kappa_{n}}
      \leq D_\kappa e^{-\gamma n} \norm{f}_{H^\alpha} \norm{g}_{H^\alpha} \,,
      \quad\text{almost surely}\,.
  \end{equation}
  Moreover, there exists a finite constant $\bar D_q$ (independent of~$\kappa$) such that
  \begin{equation}
    \E D_\kappa^q \leq \bar D_q\,,
  \end{equation}
  for all sufficiently small~$\kappa \geq 0$.
\end{lemma}
\begin{remark}
  By duality, \eqref{e:X-exp-mix-disc-time} is equivalent to
  \begin{equation}\label{e:H-1-decay}
    \tag{\ref*{e:X-exp-mix-disc-time}$'$}
    \norm{g \circ X^\kappa_n }_{H^{-\alpha}} \leq D_\kappa e^{-\gamma t} \norm{g}_{H^\alpha} \,,
  \end{equation}
  for every mean-zero test function $g \in H^{-\alpha}$.
\end{remark}

Momentarily postponing the proof of Lemma~\ref{l:X-exp-mix}, we prove Theorems~\ref{t:EDIntro} and~\ref{t:umix}.

\begin{proof}[Proof of Theorem~\ref{t:EDIntro}]
  For notational convenience, we will use $\rho_t$ to denote $\rho(\cdot, t)$, the slice of~$\rho$ at time~$t$.
  Fix $\alpha > 0$, and let $n \in \N$ be a large time.
  Assumption~\ref{a:flow} and parabolic regularity implies
  \GI[2024-01-01]{TODO: Find a reference}
  \begin{equation}\label{e:H1L1}
    \norm{\rho_{n} - \bar \rho}_{L^\infty} \leq \frac{C}{\kappa^{\frac{2\alpha + d}{4}}} \norm{\rho_{n-1} - \bar \rho}_{H^{-\alpha}} \,.
  \end{equation}
  To bound the right hand side, we note that the Kolmogorov backward equation implies that for any~$t \geq 1$ we have
  \begin{equation}
    \rho_{t+1}(x) = \E_W \rho_1 \circ X^{\kappa}_{1, 1+t}(x)\,.
  \end{equation}
  Here $\E_W$ denotes the expectation with respect to the $\P_W$ marginal of the product measure~$\P = \P_W \otimes \P_0$, and $X^\kappa_{1, \cdot}$ is the solution of the SDE
  \begin{equation}\label{e:SDEX1}
    dX^{\kappa}_{1, t}(x) = - u(X^{\kappa}_{1, t}(x), t) \, dt + \sqrt{2\kappa} \, dW_t\,,
    \quad
    X^{\kappa}_{1, 1}(x) = x \,.
  \end{equation}

  Since the distribution of~$u$ is time homogeneous, we may apply Lemma~\ref{l:X-exp-mix} to~$X^\kappa_{1, \cdot}$.
  Thus, using~\eqref{e:H-1-decay} yields
  \begin{align}
    \norm{\rho_{n-1} - \bar \rho}_{H^{-\alpha}}
      &= \norm{\E_W \rho_{1} \circ X^\kappa_{1, n-1}  - \bar \rho}_{H^{-\alpha}}
      \leq \E_W \norm{\rho_{1} \circ X^\kappa_{1, n-1} - \bar \rho}_{H^{-\alpha}}
    \\
      &\leq D_\kappa e^{-\gamma (n-1)}
	\norm{\rho_{1} - \bar \rho}_{H^\alpha}\,.
  \end{align}
  Here we used the fact that since $D_\kappa$ is independent of~$W$ which implies $D_\kappa = \E_W D_\kappa$.
  Finally, we note that parabolic regularity implies
  \begin{equation}
    \norm{\rho_1 - \bar \rho}_{H^\alpha}
      \leq \frac{C}{\kappa^{ \frac{2\alpha + d}{4}}} \norm{\rho_0 - \bar \rho}_{L^1}\,.
  \end{equation}

  Combining the above, we note
  \begin{equation}\label{e:rho-3t-decay}
    \norm{\rho_{n} - \bar \rho}_{L^\infty}
      \leq \frac{C e^{-\gamma (n-1)} D_\kappa}{\kappa^{\alpha + d/2}} \norm{\rho_{0} - \bar \rho}_{L^\infty} \,,
  \end{equation}
  for all integer times~$n$.
  Since the $L^\infty$ norm is non-increasing, we can increase $C$ by a factor of $e^{\gamma}$ and ensure~\eqref{e:rho-3t-decay} holds for all $t \geq 0$, concluding the proof.
\end{proof}

Theorem~\ref{t:umix} is equivalent to Theorem~\ref{t:EDIntro} by a standard duality argument.
\begin{proof}[Proof of Theorem~\ref{t:umix}]
  Let~$\rho$ be a solution to~\eqref{e:ad} with initial data~$\rho_0$.
  By the Kolmogorov backward equation
  \begin{equation}
    \rho_t(x) = \E_W \rho_0 \circ X^\kappa_t(x)
      = \int_{\T^d} p^\kappa_t( x, y) \rho_0(y) \, dy
      \,.
  \end{equation}
  Thus
  \begin{equation}
    \sup_{x \in \T^d} \int_{\T^d}
      (p^\kappa_t(x, y) - 1) \rho_0(y) \, dy \, dx
      \leq \norm{\rho_t - \bar \rho}_{L^\infty}
      \leq \frac{D_\kappa}{\kappa^{\frac{d}{2} + \alpha}} e^{-\gamma t} \norm{\rho_0 - \bar \rho}_{L^1}
  \end{equation}
  Since~$\rho_0 \in L^1$ is arbitrary, this implies
  \begin{equation}
    \sup_{x \in \T^d} \norm{ p^\kappa_t( x, \cdot ) - 1}_{L^\infty}
      \leq \frac{D_\kappa}{\kappa^{\frac{d}{2} + \alpha}} e^{-\gamma t}
    \,,
  \end{equation}
  which immediately yields~\eqref{e:tmix-bound} as desired.
\end{proof}

The proofs of Lemmas~\ref{l: uge} and~\ref{l:X-exp-mix} follow quickly from existing results, and the bulk of the remainder of the paper is devoted to proving Lemma~\ref{l:harrisassumptions}.
The proof can naturally be split into two parts -- showing~$V$ is a Lyapunov function, and producing a~$\kappa$-independent small set.
We do each of these parts in Section~\ref{s:Lyapunov} and~\ref{s:small-set} respectively, and then prove Lemma~\ref{l:harrisassumptions} in Section~\ref{s:verify-harris-assumptions}.

\section{The existence of a \texorpdfstring{$\kappa$-}{k-}independent Lyapunov function.}\label{s:Lyapunov}

The goal of this section is to produce the Lyapunov function~$V$ used in Lemma~\ref{l:harrisassumptions} in Section~\ref{s:verify-harris-assumptions}.
For this we use the function~$V$ in Assumptions~\ref{a:Lyapunov} and the fact that~$u \in C^2$ (Assumption~\ref{a:flow}) to show that for sufficiently small~$\kappa$, the function~$V$ is still a Lyapunov function for~$P_\kappa^{(2)}$.

\begin{lemma}[Existence of a $\kappa$-independent Lyapunov function]\label{l:LyapunovKappa}
  Suppose Assumptions~\ref{a:Lyapunov} and~\ref{a:flow} hold.
  Then there exists $\gamma_1\in (0,1)$, $b>0$ and $V:\T^{d, (2)}\to [1,\infty)$, all independent of $\kappa$, such that
  \begin{equation}\label{e:LyapunovKappa}
    P^{(2)}_\kappa V \leq \gamma_1 V + b
  \end{equation}
  holds for all sufficiently small~$\kappa \geq 0$.
\end{lemma}

The main idea behind the proof of Lemma~\ref{l:LyapunovKappa} is that the difference of four terms $X_1^\kappa(x), X_1^\kappa(y), X_1^0(x), X_1^0(y)$ can be estimated small when $\kappa$ is close to $0$ and $x$ and $y$ are close. To carry out the details, we need to first lift all the processes to~$\R^d$.
We identify the torus $\T^d$ with the set of equivalence classes $\set{ [\check x] \st \check x \in \R^d }$,
where $[\check x]$ denotes the equivalence class of $\check x$ modulo~$\Z^d$.
Notice that for any $\check x, \check y \in \R^d$ such that $\abs{\check x - \check y}<1/2$, we have
\begin{equation}
  d(x, y) = \abs{\check x - \check y}
  \quad\text{and}\quad
  \exp_x^{-1}y = \check y - \check x
  \,,
  \qquad\text{where}\quad
  x = [\check x],~ y= [\check y]
  \,.
\end{equation}
We will implicitly identify periodic functions on~$\R^d$ with functions on the torus.

We also define the Brownian motion~$W$ on $\T^d$ by choosing a standard Brownian motion~$\check W$ on $\R^d$, and setting
\begin{equation}
  W_t = \check W_t \pmod{\Z^d}
  \,.
\end{equation}

Now, let~$\check X$ be the solution to~\eqref{e:SDEX} on~$\R^d$ with the Brownian motion~$\check W$, and notice $X_t = [\check X_t]$.
To prove Lemma~\ref{l:LyapunovKappa} we need two elementary estimates on the flows~$\check X^\kappa$, which we state below.

\begin{lemma}\label{l:gronwall-difference}
  There exists a constant $C_0 = C_0( \sup_{\xi, t} \norm{\mathscr U}_{C^1} )\geq 1$ such that for every $\kappa \geq 0$ and $\check x, \check y \in \R^d$, we have
  \begin{equation} \label{e: grownwall-bound}
    \frac{\abs{\check x - \check y}}{C_0}
      \leq \abs{\check X^\kappa_1(\check x) - \check X^\kappa_1(\check y)  }
      \leq C_0 \abs{\check x - \check y}
    \,,
    \quad\text{almost surely}
    \,.
  \end{equation}
\end{lemma}
\begin{lemma}\label{l: fourtermsdiff}
  For each $\check x, \check y \in \R^d$ and $t\in [0,1]$, define $\varrho_t^\kappa(\check x, \check y)$ by
  \begin{equation}
    \varrho^\kappa_t(\check x, \check y) \defeq
      \abs[\big]{\check X_t^\kappa(\check x) - \check X_t^\kappa(\check y) - (\check X_t^0(\check x)- \check X_t^0(\check y))}
    \,.
  \end{equation}
  There exists a constant $C_1  = C_1( \sup_{\xi, t} \norm{\mathscr U}_{C^2} )$  such that for any $\alpha>0$ and $\kappa > 0$, we have
  \begin{equation}\label{e: fourtermsdiff}
    \one_{ \set{\check W_1^* \leq \alpha } } \varrho_1^\kappa(\check x, \check y)
    \leq C_1(\alpha \sqrt{\kappa} + \abs{\check x- \check y})\abs[\big]{\check X_1^0(\check x)- \check X_1^0(\check y)}
    \,.
  \end{equation}
  Here $\check W^*$ is the running maximum of $\abs{\check W}$, defined by
  \begin{equation}
    \check W^*_t = \max_{s \leq t} \abs{\check W_s}\,.
  \end{equation}
\end{lemma}

Momentarily postponing the proofs of Lemmas~\ref{l:gronwall-difference} and~\ref{l: fourtermsdiff}, we now prove Lemma~\ref{l:LyapunovKappa}.

\begin{proof}[Proof of Lemma~\ref{l:LyapunovKappa}]

We prove this proposition in two following steps.
First, we show that for any $c > 0$, there exists a constant $\epsilon < s_*$ such that
\begin{equation} \label{e: SupSupP2V}
\limsup_{\kappa \to 0^+} \sup_{(x,y)\in \Delta(\epsilon)} \frac{\abs{P^{(2)}_\kappa V-P_0^{(2)}V}}{V} < c \,.
\end{equation}
Then, using~\eqref{e:LyapunovEpEq0} and choosing $c < 1 - \tilde \gamma$ imply that for some $\epsilon>0$ and all sufficiently small $\kappa\geq 0$,
\begin{equation}\label{e:Lyapunov-Pkappa-DeltaEp}
  P_\kappa^{(2)} V < (\tilde{\gamma} + c) V \quad\text{on } \Delta(\epsilon)\,.
\end{equation}
Outside~$\Delta(\epsilon)$, we will show that there exists a constant $b>0$ such that
\begin{equation} \label{e: P2KappaVLeK2}
   P_\kappa^{(2)} V \leq b
   \quad\text{on } \Delta(\epsilon)^c\,,
\end{equation}
for all sufficiently small $\kappa\geq 0$.
Using~\eqref{e:Lyapunov-Pkappa-DeltaEp} and~\eqref{e: P2KappaVLeK2} immediately implies~\eqref{e:LyapunovKappa} as desired.

In order to finish the proof we need to prove~\eqref{e: SupSupP2V} and~\eqref{e: P2KappaVLeK2}.
To prove~\eqref{e: SupSupP2V}, let $\epsilon$ be a  small $\kappa$-independent constant that will be chosen later. We let $x, y \in \Delta(\epsilon)$ and note
\begin{equation}
 P^{(2)}_\kappa V(x,y)-P^{(2)}_0 V(x,y) = \E[V(X_1^\kappa(x), X_1^\kappa(y))- V(X_1^0(x), X_1^0(y))]\,.
\end{equation}

Choose $\epsilon > 0$ such that for the constant $C_0$ in~\eqref{e: grownwall-bound} and~$s_*$ from Assumption~\ref{a:Lyapunov}, we have $C_0\epsilon < s_*$.
Let $\check x, \check y \in \R^d$ such that $[\check x]=x$, $[\check y]=y$, and $d(x,y)=\abs{\check x - \check y}$. Then by Lemma~\ref{l:gronwall-difference}, we have
\begin{equation}
  \abs{\check X_1^\kappa(\check x)-\check X_1^\kappa(\check y)} < C_0 \epsilon < s_* < \frac{1}{2}
  \,,
\end{equation}
for all $\kappa\geq 0$, and hence
\begin{equation}\label{e: X1kappaLipschitz}
\frac{1}{C_0}d(x,y)\leq d(X_1^\kappa(x), X_1^\kappa(y)) = \abs{\check X_1^\kappa(\check x)- \check X_1^\kappa(\check y)} \leq C_0d(x,y) \leq s_*
\,.
\end{equation}

Then using~\eqref{e: LyapunovV}, we get
\begin{equation}
P^{(2)}_\kappa V(x,y)-P^{(2)}_0 V(x,y) = \E\Bigg[\frac{\psi^\kappa}{\abs{\check Z_1+\tilde Z_1^\kappa}^p}-\frac{\psi^0}{\abs{\check Z_1}^p}\Bigg]\,,
\end{equation}
where
\begin{gather}
\psi^\kappa \defeq \psi(X_1^\kappa(x), \hat \omega(X_1^\kappa(x), X_1^\kappa(y)))\,,\\
\psi^0 \defeq \psi(X_1^0(x), \hat \omega(X_1^0(x), X_1^0(y)))\,,\\
Z_1 \defeq \check X_1^0(\check x) - \check X_1^0(\check y)\,,\\
\tilde Z_1^\kappa \defeq \check X_1^\kappa(\check x) - \check X_1^\kappa(\check y)-(\check X_1^0(\check x) - \check X_1^0(\check y))\,,
\\
\hat \omega(x', y') = \frac{\exp^{-1}_{x'}(y')}{d(x', y')}
\,.
\end{gather}
In particular, we can rewrite the terms into
\begin{equation}
\frac{\psi^\kappa}{\abs{\check Z_1+\tilde Z_1^\kappa}^p}-\frac{\psi^0}{\abs{\check Z_1}^p} = \psi^\kappa \paren[\Big]{\frac{1}{\abs{\check Z_1+\tilde Z_1^\kappa}^p}-\frac{1}{\abs{\check Z_1}^p}} + \frac{1}{\abs{\check Z_1}^p}(\psi^\kappa - \psi^0)\,.
\end{equation}
and deduce
\begin{align}
  \frac{\abs{P^{(2)}_\kappa V-P^{(2)}_0 V}}{V}
    &\leq \frac{d(x,y)^p}{\inf_{SM}\psi_p}\paren[\Big]{\E\Bigl[\psi^\kappa\abs[\Big]{\frac{1}{\abs{\check Z_1+\tilde Z_1^\kappa}^p}-\frac{1}{\abs{\check Z_1}^p}} + \frac{1}{\abs{\check Z_1}^p}\abs{\psi^\kappa - \psi^0}\Bigr]}
  \\
    &\leq \frac{d(x,y)^p}{\inf_{SM}\psi}\Bigl(
      \E \one_{\set{\check W_1^* \leq \alpha}} (F_1 + F_2)
      + \P \paren{\check W_1^* \geq \alpha} \sup_\Omega (F_1 + F_2)
      \Bigr)\,,
    \label{e:P2kappa-P201}
\end{align}
where
\begin{align}
  F_1 \defeq \psi^\kappa\abs[\Big]{\frac{1}{\abs{\check Z_1+\tilde Z_1^\kappa}^p}-\frac{1}{\abs{\check Z_1}^p}} \,,
  \quad\text{and}\quad
  F_2 \defeq \frac{1}{\abs{\check Z_1}^p}\abs{\psi^\kappa - \psi^0}\,,
\end{align}
and~$\alpha > 0$ is a large $\kappa$-independent constant that will be chosen shortly.

We will now bound each term on the right of~\eqref{e:P2kappa-P201}.
On the event $\set{\check W_1^* \leq \alpha}$, we note that Lemma~\ref{l: fourtermsdiff} implies
\begin{equation} \label{e: wub}
\abs{\tilde Z_1^\kappa} \leq C_1(\alpha \sqrt{\kappa} + \epsilon)\abs{\check Z_1}\,,
\end{equation}
so
\begin{equation} \label{e: zpluswlb}
\abs{\check Z_1+\tilde Z_1^\kappa} \geq \abs{\check Z_1} - \abs{\tilde Z_1^\kappa} \geq \abs{\check Z_1}(1-C_1(\alpha\sqrt{\kappa}+\epsilon)) > 0\,,
\end{equation}
for sufficiently small~$\kappa , \epsilon > 0$.
Moreover,
\begin{equation}
\abs{\check Z_1+\tilde Z_1^\kappa} \leq \abs{\check Z_1}+\abs{\tilde Z_1^\kappa} \leq \abs{\check Z_1}(1+C_1(\alpha\sqrt{\kappa}+\epsilon))\,.
\end{equation}
These two inequalities imply that
\begin{gather}
  \label{e:zwtmp1}
  \frac{1}{\abs{\check Z_1}^p} - \frac{1}{\abs{\check Z_1+\tilde Z_1^\kappa}^p} \leq \frac{1}{\abs{\check Z_1}^p}\paren[\Big]{1-\frac{1}{(1+C_1(\alpha\sqrt{\kappa}+\epsilon))^p}}\,,
  \\
  \label{e:zwtmp2}
  \frac{1}{\abs{\check Z_1+\tilde Z_1^\kappa}^p} - \frac{1}{\abs{\check Z_1}^p} \leq \frac{1}{\abs{\check Z_1}^p}\paren[\Big]{\frac{1}{(1-C_1(\alpha\sqrt{\kappa}+\epsilon))^p}-1}\,.
\end{gather}
By convexity of the function $\xi \mapsto \xi^{-p}$, we have
\begin{equation}
1-\frac{1}{(1+C_1(\alpha\sqrt{\kappa}+\epsilon))^p} \leq
\frac{1}{(1-C_1(\alpha\sqrt{\kappa}+\epsilon))^p}-1 \,.
\end{equation}
Combined with~\eqref{e:zwtmp1} and~\eqref{e:zwtmp2} this implies
\begin{equation} \label{e: negativePpowerdiff}
\abs[\Big]{\frac{1}{\abs{\check Z_1+\tilde Z_1^\kappa}^p}-\frac{1}{\abs{\check Z_1}^p}} \leq \frac{1}{\abs{\check Z_1}^p}\paren[\Big]{\frac{1}{(1-C_1(\alpha\sqrt{\kappa}+\epsilon))^p}-1}\,.
\end{equation}
Multiplying~\eqref{e: negativePpowerdiff} by $\frac{d(x,y)^p}{\inf_{SM}\psi} \psi^\kappa$ and using~\eqref{e: X1kappaLipschitz} gives
\begin{equation}\label{t: first}
 \frac{d(x,y)^p}{\inf_{SM}\psi}  \E \one_{ \set{\check W_1^* \leq \alpha} } F_1
    \leq C_0^p\frac{\norm{\psi}_\infty}{\inf_{SM}\psi}\paren[\Big]{\frac{1}{(1-C_1(\alpha\sqrt{\kappa}+\epsilon))^p}-1}
    \,.
\end{equation}
\smallskip

Next, we bound $\E \one_{ \set{\check W_1^* \leq \alpha} } F_2$.
We note that by~\eqref{e: xk0grownwall},~\eqref{e: xxygrownwall}, and~\eqref{e: X1kappaLipschitz}, we can choose sufficiently small $\kappa, \alpha, \epsilon$ to ensure
\begin{equation}
  \diam\set{\check X_1^\kappa(\check x), \check X_1^0(\check x), \check X_1^\kappa(\check y), \check X_1^0(\check y)} < \frac{1}{2}
  \,.
\end{equation}
In this case we have
\begin{align}
d((X_1^\kappa(x),&\hat\omega(X_1^\kappa(x),X_1^\kappa(y))), (X_1^0(x),\hat\omega(X_1^0(x),X_1^0(y)))) \\
&= \abs{\check X_1^\kappa (\check x) - \check X_1^0(\check x)} + \abs[\Bigg]{\frac{\check X_1^\kappa(\check y) - \check X_1^\kappa(\check x)}{\abs{\check X_1^\kappa(\check y) - \check X_1^\kappa(\check x)}}-\frac{\check X_1^0(\check y) - \check X_1^0(\check x)}{\abs{\check X_1^0(\check y) - \check X_1^0(\check x)}}}\\
&= \abs{\check X_1^\kappa (\check x) - \check X_1^0(\check x)} + \abs[\Bigg]{
\frac{\check Z_1+\tilde Z_1^\kappa}{\abs{\check Z_1+\tilde Z_1^\kappa}} - \frac{\check Z_1}{\abs{\check Z_1}}}\,.
\end{align}

The inequalities~\eqref{e: xk0grownwall} and the general bound
\begin{equation}
\abs[\Big]{\frac{v}{\abs{v}}-\frac{v'}{\abs{v'}}} \leq \frac{2\abs{v-v'}}{\min\set{\abs{v},\abs{v'}}}\,,
\end{equation}
immediately imply
\begin{align}\label{e:inputs-difference}
  \one_{ \set{\check W_1^* \leq \alpha}}
\paren[\Bigg]{
  \abs{\check X_1^\kappa (\check x) - \check X_1^0(\check x)} + \abs[\Bigg]{
\frac{\check Z_1+\tilde Z_1^\kappa}{\abs{\check Z_1+\tilde Z_1^\kappa}} - \frac{\check Z_1}{\abs{\check Z_1}}}
}\\
  \leq
  C_0\alpha\sqrt{\kappa} + \frac{2 \abs{\tilde Z_1^\kappa}}{\min\set{\abs{\check Z_1+\tilde Z_1^\kappa}, \abs{\check Z_1}}}\,.
\end{align}
Using~\eqref{e: wub} and~\eqref{e: zpluswlb} we see
\begin{equation}\label{e:tmp2024-01-31-1}
  C_0\alpha\sqrt{\kappa} + 2\frac{\abs{\tilde Z_1^\kappa}}{\min\set{\abs{\check Z_1+\tilde Z_1^\kappa}, \abs{\check Z_1}}}
  \leq C_0\alpha\sqrt{\kappa} + \frac{2C_1(\alpha\sqrt{\kappa}+\epsilon)}{1-C_1(\alpha\sqrt{\kappa}+\epsilon)} \,.
\end{equation}
Now fix $\eta > 0$ to be a $\kappa$-independent constant that will be chosen later.
Using uniform continuity of~$\psi$ find~$\delta > 0$ such that
\begin{equation} \label{e: psipuc}
  \forall z,z'\in SM, \quad d(z,z')<\delta \implies \abs{\psi(z)-\psi(z')} < \eta\,.
\end{equation}
If~$\kappa, \epsilon$ are sufficiently small, the right hand side of~\eqref{e:tmp2024-01-31-1} can be made smaller than~$\delta$.
Using~\eqref{e:inputs-difference} implies
\begin{equation}
  d((X_1^\kappa(x),\hat\omega(X_1^\kappa(x),X_1^\kappa(y))), (X_1^0(x),\hat\omega(X_1^0(x),X_1^0(y)))) < \delta\,,
\end{equation}
and using~\eqref{e: psipuc} implies
\begin{equation}
  \abs{\psi^\kappa - \psi^0} < \eta\,.
\end{equation}
Thus by~\eqref{e: X1kappaLipschitz}, we see
\begin{equation}\label{t: second}
 \frac{d(x,y)^p}{\inf_{SM}\psi}  \E \one_{ \set{\check W_1^* \leq \alpha} } F_2
  \leq
   \frac{C_0^p}{\inf_{SM}\psi}\eta\,.
\end{equation}
\smallskip

Finally, for the last term on the right of~\eqref{e:P2kappa-P201}, we note
\begin{multline}
\psi^\kappa\abs[\Big]{\frac{1}{\abs{\check Z_1+\tilde Z_1^\kappa}^p}-\frac{1}{\abs{\check Z_1}^p}} + \frac{1}{\abs{\check Z_1}^p}\abs{\psi^\kappa - \psi^0} \\
\leq \norm{\psi}_\infty \paren[\Big]{\frac{1}{\abs{\check Z_1+\tilde Z_1^\kappa}^p}+\frac{1}{\abs{\check Z_1}^p}} + \frac{2}{\abs{\check Z_1}^p}\norm{\psi}_\infty\,,
\end{multline}
so
using~\eqref{e: X1kappaLipschitz} shows
\begin{equation}\label{t: third}
  \frac{d(x,y)^p}{\inf_{SM}\psi}
    \P \paren{\check W_1^* \geq \alpha} \sup_\Omega (F_1 + F_2)
  \leq \frac{4C_0^p \norm{\psi}_\infty}{\inf_{SM}\psi}\P[\check W_1^* > \alpha]
  \,.
\end{equation}
\smallskip

Using~\eqref{t: first}, \eqref{t: second} and~\eqref{t: third} in~\eqref{e:P2kappa-P201} we obtain
\begin{multline}
\sup_{(x,y) \in \Delta(\epsilon)}\frac{\abs{P^{(2)}_\kappa V-P^{(2)}_0 V}}{V} \leq C_0^p\frac{\norm{\psi}_\infty}{\inf_{SM}\psi}\paren[\Big]{\frac{1}{(1-C_1(\alpha\sqrt{\kappa}+\epsilon))^p}-1} \\
+ \frac{C_0^p}{\inf_{SM}\psi}\eta + \frac{4C_0^p \norm{\psi}_\infty}{\inf_{SM}\psi}\P[W_1^* > \alpha]\,.
\end{multline}

Thus,
\begin{multline}\label{e:tmp2024-01-31-2}
\limsup_{\kappa \to 0+} \sup_{(x,y)\in \Delta(\epsilon)} \frac{\abs{P^{(2)}_\kappa V-P_0^{(2)}V}}{V} \leq C_0^p\frac{\norm{\psi}_\infty}{\inf_{SM}\psi}\paren[\Big]{\frac{1}{(1-C_1\epsilon)^p}-1} \\
+ \frac{C_0^p}{\inf_{SM}\psi}\eta + \frac{4C_0^p \norm{\psi}_\infty}{\inf_{SM}\psi}\P[W_1^* > \alpha] \,.
\end{multline}
Now, we choose $\alpha, \eta,$ and $\epsilon$ such that
\begin{gather}
\frac{4C_0^p \norm{\psi}_\infty}{\inf_{SM}\psi}\P[W_1^* > \alpha] < \frac{1}{3}c\,, \\
 \frac{C_0^p}{\inf_{SM}\psi}\eta < \frac{1}{3}c\,,\\
C_0^p \frac{\norm{\psi}_\infty}{\inf_{SM}\psi}(\frac{1}{(1-C_1\epsilon)^p}-1) < \frac{1}{3}c \,,
\end{gather}
then using~\eqref{e:tmp2024-01-31-2} will imply~\eqref{e: SupSupP2V} as desired.
\medskip

Finally, in order to prove~\eqref{e: P2KappaVLeK2}, we see from the Assumption~\ref{a:Lyapunov} that $V$ is continuous on the compact set $K \defeq \overline{\Delta(s_*)}\setminus \Delta(\frac{s_*}{2})$ so it can be continuously extended to the compact set $K' \defeq \T^{d, (2)}\setminus \Delta(\frac{s_*}{2})$ such that
\begin{equation}
  1 \leq \inf_{K} V
    = \inf_{K'} V
    \leq \sup_{K'} V
    = \sup_{K} V
    \,.
\end{equation}

Now, let $(x,y) \in \Delta(\epsilon)^c$. On the event $E_1\defeq \set{(X_1^\kappa(x), X_1^\kappa(y))\in \Delta(s_*)}$, there exist $\check x,\check y \in \R^d$ such that $[\check x] = x$, $[\check y] = y$, and
\begin{equation}
d(X_1^\kappa(x), X_1^\kappa(y)) = \abs{\check X_1^\kappa(\check x)- \check X_1^\kappa(\check y)} < s_*\,.
\end{equation}
By~\eqref{e: grownwall-bound}, $\abs{\check x - \check y}<\frac{s_*}{C_0}<\frac{1}{2}$ so $\abs{\check x - \check y}=d(x,y)$ and
\begin{equation}
d(X_1^\kappa(x), X_1^\kappa(y)) \geq C_0d(x,y) \geq C_0\epsilon\,,
\end{equation}
which implies
\begin{equation}\label{e:VE1}
\E[V(X_1^\kappa(x), X_1^\kappa(y))\one_{E_1}]\leq C_0^{-p}\epsilon^{-p}\norm{\psi}_\infty\,.
\end{equation}
On the event $E_1^c$,
we have
\begin{equation}\label{e:VE1c}
  \E[V(X_1^\kappa(x), X_1^\kappa(y))\one_{E_1^c}]
    \leq \sup_{K'} V
  \,.
\end{equation}
Bounds~\eqref{e:VE1} and~\eqref{e:VE1c} yield~\eqref{e: P2KappaVLeK2} with
\begin{equation}
  b\defeq C_0^{-p}\epsilon^{-p}\norm{\psi}_\infty + \sup_{K'} V
  \,.
\end{equation}
This completes the proof.
\end{proof}

\subsection*{Gradient estimates on the stochastic flows}

It remains to prove  Lemmas \ref{l:gronwall-difference} and~\ref{l: fourtermsdiff}. Lemma~\ref{l:gronwall-difference} is a direct result of applying Gr\"onwall's inequality twice to the difference $\abs{X_t^\kappa(\check x)-X_t^\kappa(\check y)}$ near $t=0$ and $t=1$.

\begin{proof}[Proof of Lemma~\ref{l:gronwall-difference}]
For the upper bound, we see that for any $\kappa\geq 0$ and $0\leq t \leq 1$,
\begin{equation}
\check X_t^\kappa(\check x) - \check X_t^\kappa(\check y) = \check x- \check y + \int_0^t u(X_s^\kappa(\check x),s) - u(\check X_s^\kappa(\check y),s) ds\,,
\end{equation}
so
\begin{equation}
\abs{\check X_t^\kappa(\check x) - \check X_t^\kappa(\check y)} \leq \abs{\check x- \check y} + A_1 \int_0^t \abs{\check X_s^\kappa(\check x) - \check X_s^\kappa(\check y)} ds\,.
\end{equation}
By Gr\"onwall's inequality, we have
\begin{equation}
\abs{\check X_1^\kappa(x) - \check X_1^\kappa(y)} \leq e^{A_1} \abs{\check x - \check y}\,,
\end{equation}
where
\begin{equation}\label{e: A1}
  A_1\defeq
    \norm{\nabla_x \mathscr U}_{L^\infty(\mathscr M \times \T^d \times [0, 1] )}
\end{equation}
Similarly, for the lower bound, we see that for any $\kappa\geq 0$ and $0\leq t\leq 1$,
\begin{align}
\check X_{1-t}^\kappa (\check x) - \check X_{1-t}^\kappa (\check y) &= \check X_1^\kappa (\check x) - \check X_1^\kappa(\check y) - \int_{1-t}^1 \paren[\Big]{u(\check X_s^\kappa(\check x),s)- u(\check X_s^\kappa(\check y),s)} ds \\
&= \check X_1^\kappa (\check x) - \check X_1^\kappa(\check y) - \int_0^t \paren[\Big]{u(\check X_{1-s}^\kappa(\check x),s)- u(\check X_{1-s}^\kappa(\check y),s)} ds
\end{align}
so
\begin{equation}
\abs{\check X_{1-t}^\kappa(\check x) - \check X_{1-t}^\kappa(\check y)} \leq \abs{\check X_1^\kappa (\check x) - \check X_1^\kappa(\check y)} + A_1 \int_0^t \abs{\check X_{1-s}^\kappa(\check x) - \check X_{1-s}^\kappa(\check y)} ds\,.
\end{equation}
By Gr\"onwall's inequality, we have
\begin{equation}
\abs{\check x- \check y} \leq e^{A_1} \abs{\check X_1^\kappa(\check x) - \check X_1^\kappa(\check y)}\,.
\end{equation}
Thus, setting $C_0\defeq e^{A_1}$ concludes the proof.
\end{proof}

To prove Lemma~\ref{l: fourtermsdiff}, we first explicitly write down the differences $\check X_t^\kappa(\check x)-\check X_t^\kappa(\check y)$ and $\check X_t^0(\check x)- \check X_t^0(\check y)$ by using the differential equations they satisfy. Then, we use mean value theorem and Gr\"onwall's inequality multiple times to estimate their difference.
\begin{proof}[Proof of Lemma~\ref{l: fourtermsdiff}]
For $0\leq t\leq 1$, we have
\begin{align} \label{e: xtk diff}
\check X_t^\kappa(\check x) - \check X_t^\kappa(\check y)
  &= \check x - \check y + \int_0^t (u(\check X_s^\kappa(\check x),s) - u(\check X_s^\kappa(\check y)),s) \, ds \\
  &= \check x - \check y + \int_0^t \nabla_x u(\beta_1(s),s)\cdot (\check X_s^\kappa(\check x) - \check X_s^\kappa(\check y)) \, ds
\end{align}
for some $\beta_1(s) = (\beta_1^1(s), \beta_1^2(s), \ldots, \beta_1^d(s)) \in (\R^d)^d$, where
\begin{equation}
\nabla_x u(\beta_1(s), s) \defeq (\nabla u_1(\beta_1^1(s), s), \nabla u_2(\beta_1^2(s), s), \ldots, \nabla u_d(\beta_1^d(s), s)) \in (\R^d)^d\,.
\end{equation}
Similarly,
\begin{equation} \label{e: xt0 diff}
\check X_t^0(\check x)-\check X_t^0(\check y) = \check x- \check y + \int_0^t \nabla_x u (\beta_2(s),s) \cdot (\check X_s^0(\check x) - \check X_s^0(\check y)) \, ds
\end{equation}
for some $\beta_2(s) \in (\R^d)^d$.
We define
\begin{equation}
S(t) \defeq \check X_t^\kappa(\check x) - \check X_t^\kappa(\check y) - (\check X_t^0(\check x) - \check X_t^0(\check y))\,.
\end{equation}
Then, by taking the difference between~\eqref{e: xtk diff} and~\eqref{e: xt0 diff}, we get
\begin{multline}
S(t) = \int_0^t \nabla_x u(\beta_1(s),s) \cdot S(s) \\
+ (\nabla_x u (\beta_1(s),s) - \nabla_x u(\beta_2(s),s)) \cdot (\check X_s^0(\check x) - \check X_s^0 (\check y)) \, ds \,,
\end{multline}
which implies
\begin{equation}
\varrho_t^\kappa(\check x,\check y)\leq A_1 \int_0^t \varrho_s^\kappa(\check x,\check y) ds + A_2 \int_0^t \abs{\beta_1(s) - \beta_2(s)}\abs[\Big]{\check X_s^0(\check x) - \check X_s^0(\check y)} \, ds\,,
\end{equation}
where
\begin{equation}
  A_2\defeq
    \norm{\nabla_x^2 \mathscr U}_{L^\infty(\mathscr M \times \T^d \times [0, 1] )}
  \,.
\end{equation}
  For each $(s, i)\in [0, t]\times \set{1, \ldots, d}$,
  the points $\beta_1^i(s)$, and~$\beta_2^i(s)$ are on the line segments joining~$\check X^0_s(\check x)$ and~$\check X^0_s(\check y)$, and~$\check X^\kappa_s(\check x)$ and~$\check X^\kappa_s(\check y)$ respectively.
 Thus
\begin{align}
  \abs{\beta_1(s)-\beta_2(s)}\leq \max\{ &\abs{\check X_s^\kappa(\check y)-\check X_s^0(\check y)}, \abs{\check X_s^\kappa(\check x)-\check X_s^0(\check x)},\\
   & \abs{\check X_s^\kappa(\check y)-\check X_s^0(\check x)}, \abs{\check X_s^\kappa(\check x)-\check X_s^0(\check y)}\}
\end{align}
and in particular,
\begin{align}
  \one_{\set{\check W_1^* \leq \alpha}} \abs{\check X_s^\kappa(\check x)-\check X_s^0(\check x)} &\leq \alpha\sqrt{\kappa}e^{A_1 s}\label{e: xk0grownwall}\\
\abs{\check X_s^0(\check x)-\check X_s^0(\check y)} &\leq \abs{\check x-\check y}e^{A_1 s}
\label{e: xxygrownwall}
\end{align}
by Gr\"onwall's inequality.

As a result, we obtain
\begin{multline}
  \one_{ \set{\check W_1^* \leq \alpha } }
    \varrho_t^\kappa(\check x,\check y)
  \leq
    A_1  \int_0^t \one_{ \set{\check W_1^* \leq \alpha } }
      \varrho_s^\kappa (\check x, \check y) ds
  \\
    + A_2  (\alpha\sqrt{\kappa}+\abs{\check x-\check y})e^{A_1 t}\int_0^t \abs[\Big]{\check X_s^0(\check x) - \check X_s^0(\check y)}ds
    \,.
\end{multline}
Using Gr\"onwall's inequality and Lemma~\ref{l:gronwall-difference} this gives
\begin{gather}
  \one_{ \set{\check W_1^* \leq \alpha } }
    \varrho_1^\kappa(\check x,\check y)
  \leq
    C_1(\alpha\sqrt{\kappa} + \abs{\check x-\check y})\abs[\Big]{\check X_1^0(\check x) - \check X_1^0(\check y)}\,,
\end{gather}
for some constant~$C_1 = C_1( A_1, A_2 )$.
\end{proof}

\section{A stable small set.}\label{s:small-set}

Next, in order to obtain the minorizing condition~\eqref{e:V-sublevel-small} stated in Lemma~\ref{l:harrisassumptions}, we will show that Assumptions~\ref{a:submersion-2point} and~\ref{a:flow} imply the existence of a $\kappa$-independent small set.

\begin{lemma}\label{l: stablesmallset}
Suppose Assumptions~\ref{a:submersion-2point} and~\ref{a:flow} hold.
Then there exist nonempty open sets $A, B \subseteq \T^{d, (2)}$ and a constant $\beta > 0$ such that
\begin{equation}\label{e:stablesmallsetconclusion}
\inf_{x\in A} P_\kappa^{(2),n}(x,\cdot) \geq \beta \Leb|_{B} (\cdot)\,.
\end{equation}
for all sufficiently small $\kappa \geq 0$.
Here~$\Leb|_B$ denotes the restriction of the Lebesgue measure to~$B$.
\end{lemma}

To prove Lemma~\ref{l: stablesmallset}, we need the following two lemmas.
\begin{lemma}\label{l:submersionPushForward}
  Let $f \colon B_r \subseteq \R^N \to \R^d$ be a $C^1$ function such that $f(0) = 0$ and $\rank(Df(0)) = d$.
  Then there exist $\delta, s > 0$ such that, for any $g \colon \R^N \to \R^d$ Lipschitz with $\|f-g\|_{L^\infty(B_r)} \leq \delta$, we have
  \begin{equation}
  	g(B_r) \supseteq B_s
  \end{equation}
  Here
  $B_r$ denotes the open ball centered at~$0$.
\end{lemma}
\begin{proof}
  By the constant rank theorem, there are diffeomorphisms $\alpha$ and $\beta$ which fix the origin such that $(\alpha \circ f \circ \beta)(x_1, \dots, x_N) = (x_1, \dots, x_d)$. Choose $\delta > 0$ so that $\|(\alpha \circ f \circ \beta) - (\alpha \circ g \circ \beta)\|_{L^\infty(B_1)} \leq \frac{1}{4}$. Let $x_{d+1}, \dots, x_N$ be arbitrary with $x_{d+1}^2 + \dots + x_N^2 \leq \frac{1}{4}$ and define $h(x_1, \dots, x_d) := (\alpha \circ g \circ \beta)(x_1, \dots, x_d, x_{d+1}, \dots, x_N)$. Then $h \colon B_{3/4} \subseteq \R^d \to \R^d$ satisfies $|h(x)-x| \leq \frac{1}{4}$ at every point.

  We claim that $h(B_{3/4}) \supseteq B_{1/2}$. Indeed, compute the degree $\deg(h, B_{3/4}, x) = 1$ for each $x \in B_{1/2}$ (the degree is a homotopy invariant, and $h$ is homotopic to the identity).

  Then $(\alpha \circ g \circ \beta)(B_1) \supseteq B_{1/2}$, and therefore $g(B_r) \supseteq B_s$ for some $s > 0$ depending only on $\alpha$ and $\beta$.
\end{proof}

\begin{lemma}\label{l:LipDependenceOnU}
  For any $n \in \N$, $\kappa \geq 0$, $x \in \T^{d, (2)}$, $\xi, \xi'\in \mathscr M^n$, and almost any realization of the noise $W$, we have
  \begin{equation}
    \label{eq:lipschitz-dependence-on-u}
    d( X_n^{\kappa, (2)}(\xi, x), X_n^{\kappa, (2)}(\xi', x) )
      \leq A_3\exp(n(1+A_1)) d_\infty(\xi, \xi') \,,
  \end{equation}
  where $A_1$ is defined as in~\eqref{e: A1}, and
  \begin{align}
    A_3 &\defeq
      \norm{\nabla_\xi \mathscr U}_{L^\infty(\mathscr M \times \T^d \times [0, 1] )}
    \,,
    \\
    d_\infty( \xi, \xi' ) &\defeq \sup_{k \leq n} d_{\mathscr M}( \xi_k, \xi_k' ) \,.
  \end{align}
\end{lemma}
\begin{proof}
  The proof follows immediately from Gr\"onwall's inequality and is very similar to the proof of Lemma~\ref{l:gronwall-difference}.
\end{proof}

With these lemmas, we can prove the existence of a $\kappa$-independent small set.
\begin{proof}[Proof of Lemma~\ref{l: stablesmallset}]
  Let $n, x_*, \xi_*, \epsilon, c, \rho_n$ be defined as in Assumption~\ref{a:submersion-2point}. Then, we see that $\mathscr X_n^{(2)} (\cdot, x_*)$ satisfies the assumption in Lemma~\ref{l:submersionPushForward} with $r=\epsilon$.
  Let $\delta, s$ be the constants given in Lemma~\ref{l:submersionPushForward} for the map $\mathscr X_n^{(2)}(\cdot, x_*)$. Now, choose $\alpha>0$ such that $\P[\check W_n^* \leq \alpha]\geq \frac{1}{2}$ and let $\eta,\kappa$ be small enough so that
  \begin{equation}
    (\eta +n\sqrt{\kappa}\alpha ) e^{nA_1} < \frac{1}{\sqrt{2}}\delta
    \,.
  \end{equation}
  Then for each choice of $x\in B_\eta(x_*)$ and a realization of Brownian path such that $\{\check W_n^* \leq \alpha\}$, we have
\begin{equation}
d(X_n^{\kappa, (2)}(\xi, x), X_n^{0, (2)}(\xi, x_*)) < \delta\,,
\end{equation}
for any $\xi \in \mathscr M^n$.
Thus, we can apply  Lemma~\ref{l:submersionPushForward} to the map
$g:\xi \mapsto X_n^{\kappa, (2)}(\xi, x)$ and see that
\begin{align}
\P_\kappa^{(2), n}(x, A) &\geq \E_W\Big[\paren[\Big]{\int_{g^{-1}(A)\cap B_\epsilon(\xi_*)}\rho_n(\xi)d\xi}\one_{\set{W_n^*\leq \alpha}}\Big]\\
&\geq c \E_W\Big[\abs{g^{-1}(A)\cap B_\epsilon(\xi_*)}\one_{\set{W^*_n \leq \alpha}}\Big]\\
\label{e:P2n-kappa}
&\geq \frac{c}{2} \abs{A \cap B_s(\mathscr X(\xi_*, x_*))} \E_W \Lip(g)^{-\dim \mathscr M}
\,.
\end{align}
Using Lemma~\ref{l:LipDependenceOnU} we see
\begin{equation}\label{e:EWLipG}
  \E_W \Lip(g)^{-\dim \mathscr M}
    \geq (A_3\exp(n(1+A_1)))^{-\dim\mathscr M}
  \,.
\end{equation}
Using~\eqref{e:EWLipG} in~\eqref{e:P2n-kappa} yields~\eqref{e:stablesmallsetconclusion} as desired.
\end{proof}

\section{Verification of the Harris Assumptions (Lemma \ref{l:harrisassumptions}).}\label{s:verify-harris-assumptions}

Given the Lyapunov function~$V$ from Lemma~\ref{l:LyapunovKappa} and the small set from Lemma~\ref{l: stablesmallset}, we will now prove Lemma~\ref{l:harrisassumptions} verifying Harris conditions for the two point chains, for all sufficiently small~$\kappa \geq 0$.

\begin{proof}[Proof of Lemma~\ref{l:harrisassumptions}]
  We first note that Lemma~\ref{l:LyapunovKappa} and induction immediately imply that for every $l \in \N$,
  \begin{equation}\label{e:iterated Lyapunov bound}
    P_\kappa^{(2), l}V \leq \gamma_{1}^l V + b_l \,,
  \end{equation}
  where
  \begin{gather}\label{e:bm}
    b_l \defeq b\sum_{i=0}^{l-1}\gamma_1^i = b \frac{1-\gamma_    1^l}{1-\gamma_1} \,.
  \end{gather}
  We will now show that~$R$ and~$l$ can be chosen so that ~\eqref{e:V-sublevel-small} is also satisfied.

  We will first show that any compact set $A \subset \T^{d, (2)}$ is a small set, uniformly in~$\kappa$.
  More precisely, we will prove that for any compact set $A\subset \T^{d, (2)}$, there exist $m\in \N$, $\beta\in (0,1)$, and a probability measure $\mu$, all independent of $\kappa$, such that
  \begin{equation}\label{e:Asmall}
    \inf_{x\in A} P_\kappa^{(2),m}(x, \cdot) \geq \beta\mu(\cdot)\,,
  \end{equation}
  for all sufficiently small $\kappa\geq 0$.
  Next, we will show that for any $R>\frac{2b}{1-\gamma_1}$, there exists a compact set $S$ such that
  \begin{equation} \label{e:V-sublevel-compact}
  \set{V\leq R} \subset S \,.
  \end{equation}

  To see why the above implies~\eqref{e:V-sublevel-small}, we first fix any $R>\frac{2b}{1-\gamma_1}$ and then find a compact set $S$ such that~\eqref{e:V-sublevel-compact} holds.
  From~\eqref{e:Asmall}, we can find $\kappa$-independent constants $l, \alpha$ and a probability measure $\nu$ such that for all sufficiently small $\kappa\geq 0$, we have
  \begin{equation} \label{e:Ssmall}
  \inf_{x\in S} P_\kappa^{(2),l}(x, \cdot) \geq \alpha\nu(\cdot) \,.
  \end{equation}
  For this particular $l\in \N$, we define
  \begin{equation}
    \gamma_3 \defeq \gamma_{1}^l,
    \quad\text{and}\quad
    K \defeq b_l \,.
  \end{equation}
  Using~\eqref{e:bm} we observe
  \begin{equation}
    R > \frac{2b}{1-\gamma_1} = \frac{2b}{1-\gamma_1^l} \frac{1-\gamma_1^l}{1-\gamma_1}
    = \frac{2K}{1 - \gamma_3}\,,
  \end{equation}
  which implies~\eqref{e:V-sublevel-small} as claimed.
  \smallskip

  \GI[2024-02-06]{Maybe give a reference for $T$-chain}
  It remains to prove~\eqref{e:Asmall} and~\eqref{e:V-sublevel-compact}.
To prove~\eqref{e:Asmall}, we first note that Assumptions~\ref{a:fts} and~\ref{a:submersion-2point}  imply $P_0^{(2)}$ is an $\psi^{(2)}$-irreducible, aperiodic $T$-chain with the property that
\begin{equation}
\psi^{(2)}(\mathcal V)>0, \quad \text{for all }  \mathcal V \subset \T^{d, (2)} \text{ open}\,.
\end{equation}
\SJ[2024-02-05]{Multiple reference here: \cite{BlumenthalCotiZelatiEA22} and~\cite{MeynTweedie09}. What would be a good way to mention them?}
\SJ[2023-08-15]{Define $\psi^{(2)}$, irreducible, aperiodic, T-chain, petite}
Assumption~\ref{a:submersion-2point} and Lemma~\ref{l: stablesmallset} imply that there exist open balls $B_r, B_
s \subset \T^{d, (2)}$ and $n\in\N$ such that for all sufficiently small $\kappa\geq 0$,
\begin{equation} \label{e:Usmall0}
\inf_{x\in B_r} P_\kappa^{(2),n}(x,\cdot)\geq \tilde \mu(\cdot)\,,
\end{equation}
where
$B_r$ and $B_s$ are open balls with radius $r, s>0$, respectively, and
\begin{equation}
  \tilde{\mu}(\cdot) \defeq \Leb(\cdot \cap B_s) \,.
\end{equation}

$B_{\frac{1}{2}r}$ is small for the chain $P_0^{(2)}$ and $\psi^{(2)}(B_{\frac{1}{2}r})>0$ so by Theorem 6.2.5 (ii) and Theorem 5.5.7 in \cite{MeynTweedie09}, we see that there exist some $q\in \N$ and $c>0$ such that
\begin{equation} \label{e:Asmall0forU}
\inf_{x\in A} P_0^{(2),q}(x,B_{\frac{1}{2}r}) \geq c > 0 \,.
\end{equation}

Then, for each $x\in A$, the measure $P_\kappa^{(2), q}(x,\cdot)$ converges weak-* to the measure $P_0^{(2), q}(x,\cdot)$ as $\kappa \to 0$ so there exists $\kappa_0(x)>0$ such that
\begin{equation}\label{e: qkappa0}
\inf_{\kappa < \kappa_0(x)}P_\kappa^{(2), q}(x, B_{\frac{1}{2}r}) \geq \frac{1}{2}P_0^{(2), q}(x, B_{\frac{1}{2}r})\,.
\end{equation}
Also, if we assume $r<1$ without loss of generality and define $A_1$ as in~\eqref{e: A1}, then for any $x\in \T^d$ and $y \in B(x, \frac{1}{2}\exp(-A_1q)r)$, we can find $\check x, \check y \in \R^d$ such that $[\check x] = x$, $[\check y] = y$, and $d(x,y)=\abs{\check x - \check y}$, and use simple Gr\"{o}wnwall bound to notice that
\begin{equation}
d(X_q^\kappa (x), X_q^\kappa(y)) = \abs{\check X_q^\kappa(\check x)-\check X_q^\kappa(\check y)} \leq \exp(A_1 q) \abs{\check x-\check y}\leq \frac{1}{2}r\,.
\end{equation}
This immediately leads to the inequality
\begin{equation}\label{e: kappaqxy}
P_\kappa^{(2),q}(y, B_r) \geq P_\kappa^{(2), q}(x, B_{\frac{1}{2}r})\,.
\end{equation}
Now, we cover the compact set $A$ with open balls $\cup_{x\in A}B(x, \frac{1}{2}\exp(-A_1q)r)$ and find a finite cover $\bigcup_{i=1}^N B(x_i, \frac{1}{2}\exp(-A_1q)r)$. If we let $\kappa < \kappa_0 \defeq \min_{i=1}^N \kappa_0(x_i)$ and $y \in A$, then by~\eqref{e: kappaqxy},~\eqref{e: qkappa0}, and~\eqref{e:Asmall0forU}, we see that
\begin{equation}
P_\kappa^{(2), q}(y, B_r)\geq \frac{1}{2}c \,.
\end{equation}
This implies for any $0\leq \kappa<\kappa_0$,
\begin{equation} \label{e: Aqkappa}
\inf_{x\in A}P_\kappa^{(2), q}(x, B_r) \geq \frac{1}{2}c\,.
\end{equation}

Defining $m\defeq n+q$ and using~\eqref{e:Usmall0} with~\eqref{e: Aqkappa} yields
\begin{equation}\label{e:P2KappaLower}
  \inf_{x\in A}P_\kappa^{(2), m}(x,\cdot) \geq  \frac{1}{2}c \tilde{\mu}(\cdot)
\end{equation}
for all sufficiently small $\kappa\geq 0$.
Then, normalizing $\frac{1}{2}c \tilde{\mu}(\cdot)$ immediately implies~\eqref{e:Asmall}, as desired.
\smallskip

Finally, to prove~\eqref{e:V-sublevel-compact}, we notice that if $s''<s_*$ and $(x,x')\in \Delta(s'')$, we have
\begin{equation}
  V(x,x')=d(x,x')^{-p}\psi_p(x, \widehat{w}(x,x')) \geq (s'')^{-p} \paren[\Big]{ \inf_{\text{SM}}\psi_p} > 0  \,.
\end{equation}
Thus making $s''>0$ sufficiently small will ensure
\begin{equation}
\Delta(s'') \subset \{V>R\}\,,
\end{equation}
which implies
\begin{equation}
\{V\leq R\} \subset \Delta(s'')^c \,.
\end{equation}
Thus $S\defeq \Delta(s'')^c$ is the desired compact set, proving~\eqref{e:V-sublevel-compact}.
This concludes the proof.
\end{proof}

\section{V-geometric ergodicity (Lemma~\ref{l: uge}).}\label{s:uge}

Given Lemma~\ref{l:harrisassumptions}, $V$-geometric ergodicity of the two point chains $P_\kappa^{(2)}$ follows directly from a theorem of Harris~\cite{Harris55,MeynTweedie09}.
The usual Harris theorem, however,  isn't quantitative enough to yield~\eqref{e: ugeGeneral} with $\kappa$-independent constants $C, \beta$.
We will instead use the version in~\cite{HairerMattingly11} which is quantitative and can be used to prove Lemma~\ref{l: uge}.

For the proof, we define the metric~$\rho_\beta$ by
\begin{equation}\label{e: rho-def-2}
  \rho_\beta(\mu_1, \mu_2) \defeq \int_{\T^{d, (2)}} (1+\beta V) \, d\abs{\mu_1-\mu_2}
  \,,
\end{equation}
where $\beta \geq 0$, $\mu_1, \mu_2$ are probability measures, and~$\abs{\mu_1 - \mu_2}$ denotes the variation of the signed measure~$\mu_1 - \mu_2$.
The quantitative Harris theorem from~\cite{HairerMattingly11} shows that $P^{(2), l}_\kappa$ is a contraction under~$\rho_\beta$, and for readers convenience we now restate this result in our context.

\begin{theorem}[Theorem 1.3 in~\cite{HairerMattingly11}]\label{t: harris}
  Make the same assumptions as in Le\-mma~\ref{l:harrisassumptions}.
  For any~$\alpha_0 \in (0, \alpha)$, and any~$\gamma_0 \in (\gamma_3 + \frac{2K}{R}, 1)$, define
  \begin{equation}
    \beta \defeq \frac{\alpha_0}{K}
    \,,
    \qquad
    \bar \alpha \defeq \max\set[\Big]{ 1 - (\alpha - \alpha_0), \frac{2 + R \beta \gamma_0}{2 + R\beta} }\,.
  \end{equation}
  Then, for any two probability measures~$\mu_1, \mu_2$ we have
  \begin{equation}\label{e:rho-contraction}
  \rho_\beta(\mu_1 P^{(2), l}_\kappa, \mu_2 P^{(2),l}_\kappa) \leq \bar \alpha \rho_\beta(\mu_1, \mu_2)\,.
  \end{equation}
\end{theorem}

Referring to~\cite{HairerMattingly11} for the proof of Theorem~\ref{t: harris}, we will now prove Lemma~\ref{l: uge}.

\begin{proof}[Proof of Lemma~\ref{l: uge}]
We will first show there exist constants $C>0$, $\gamma_2 \in (0,1)$, and $l\in \N$ such that for all sufficiently small $\kappa\geq 0$, any $\varphi:\T^{d, (2)}\to \R$ such that $\norm{\varphi}_V < \infty$, and any $m\in \N$, we have
\begin{equation}\label{e: uge}
    \norm[\Big]{ P_\kappa^{(2), lm} \varphi-\int \varphi \, d\pi^{(2)} }_V
      \leq C\gamma_2^m \norm[\Big]{ \varphi-\int \varphi \, d\pi^{(2)}}_V\,.
\end{equation}

For this, we first define a weighted norm
\begin{equation}
  \|\varphi\|_\beta \defeq \sup_x \frac{\abs{\varphi(x)}}{1+\beta V(x)}
  \,.
\end{equation}
Then, since $V\geq 1$ on $\T^{d, (2)}$, we see that the norms $\|\cdot\|_V$ and $\|\cdot\|_\beta$ are equivalent for any $\beta>0$ with
\begin{equation} \label{e: eqnorm}
\frac{1}{1+\beta}\|\varphi\|_V\leq \|\varphi\|_\beta \leq \frac{1}{\beta} \|\varphi\|_V\,.
\end{equation}
We also note that given any two probability measures $\mu_1$ and $\mu_2$ on $\T^{d, (2)}$,
\begin{gather}\label{e: rho-def}
\rho_\beta (\mu_1, \mu_2) = \sup_{\|\varphi\|_\beta \leq 1} \langle \mu_1-\mu_2, \varphi \rangle
\end{gather}
always holds, where $\ip{\mu, \varphi}$ denotes the dual pairing
\begin{equation}
  \ip{\mu, \varphi} \defeq \int_{\T^{d, (2)}} \varphi \, d\mu \,.
\end{equation}
\smallskip

Now, we're ready to prove~\eqref{e: uge}. From here on, we set $P = P_\kappa^{(2),l}$ for notational simplicity.
By Lemma~\ref{l:harrisassumptions} and Theorem~\ref{t: harris} we obtain the contraction estimate~\eqref{e:rho-contraction}.
In particular, for any $x\in \T^{d, (2)}$ and $\mu_1 \defeq \delta_x$, $\mu_2\defeq \pi^{(2)}$, we have
\begin{equation}
\rho_\beta(\delta_x P^n , \pi^{(2)}) \leq \bar{\alpha}^n \rho_\beta(\delta_x, \pi^{(2)})\,.
\end{equation}
By using~\eqref{e: rho-def} and then~\eqref{e: rho-def-2}, we notice
\begin{gather}
\frac{\abs{\langle \delta_x P^n - \pi^{(2)}, \varphi - \langle \pi^{(2)}, \varphi \rangle \rangle}}{\|\varphi - \langle \pi^{(2)}, \varphi \rangle\|_\beta} \leq \bar{\alpha}^n\rho_\beta(\delta_x, \pi^{(2)}) \leq \bar{\alpha}^n (1+\beta V(x) + \langle \pi^{(2)}, 1+\beta V \rangle)\,,
\end{gather}
and hence
\begin{gather}
\frac{\abs[\Big]{(P^n\varphi)(x)-\langle \pi^{(2)}, \varphi \rangle}}{1+\beta V(x)} \leq \bar{\alpha}^n(1+\langle \pi^{(2)}, 1+\beta V \rangle)\|\varphi - \langle \pi^{(2)}, \varphi \rangle\|_\beta\,.
\end{gather}
This holds for any $x\in \T^{d, (2)}$ so
\begin{equation}
\norm[\Big]{ P^n \varphi- \langle \pi^{(2)}, \varphi \rangle}_\beta
      \leq C_\beta \bar{\alpha}^n \|\varphi - \langle \pi^{(2)}, \varphi \rangle\|_\beta\,.
\end{equation}
Finally, using~\eqref{e: eqnorm} yields~\eqref{e: uge} where $C$ and $\gamma_2$ depend on $\alpha, \gamma_3, K, R$ but not on $\kappa$. This completes the proof for~\eqref{e: uge}.

The proof that~\eqref{e: uge} implies~\eqref{e: ugeGeneral} is a standard argument.
For any $x\in \T^{d, (2)}$ and any mean~$0$ function~$g$ we note
\begin{gather*}
	\abs[\Big]{\frac{(P_\kappa^{(2)}g)(x)}{V(x)}} = \abs[\Big]{\int_{\T^{d, (2)}}\frac{g(y)}{V(y)}\frac{V(y)}{V(x)}P_\kappa^{(2)}(x,dy)} \leq \|g\|_V\frac{(P_\kappa^{(2)}V)(x)}{V(x)} \leq (\gamma_1+b)\|g\|_V
	\,,
\end{gather*}
where $\gamma_1$ and $b$ are the constants defined in~\eqref{e:LyapunovKappa}.
Thus,
\begin{equation}
\norm[\big]{P_\kappa^{(2)}g}_V \leq (\gamma_1+b)\|g\|_V\,.
\end{equation}
This and~\eqref{e: uge} imply that for any $m\in \N$, $0\leq r < l$, \begin{equation}
\norm[\Big]{ P_\kappa^{(2), lm+r} g}_V
\leq C\max(\gamma_1+b,1)^r \gamma_2^m \norm{g}_V \leq C  \gamma_2^m \norm{g}_V\,,
\end{equation}
which proves
\begin{equation}
\norm[\Big]{ P_\kappa^{(2), n} g}_V
\leq C (\gamma_2^\frac{1}{l})^n\|g\|_V
\end{equation}
for general $n\in \N$ with possibly different constants $C$ in each line.
This completes the proof of Lemma~\ref{l: uge}.
\end{proof}

\section{Exponential Mixing of the Stochastic Flows (Lemma \ref{l:X-exp-mix}).}\label{s:bc}

In general, the geometric ergodicity of the two-point chain implies almost sure exponential mixing of the random dynamical system.
To the best of our knowledge this principle was introduced in~\cite{DolgopyatKaloshinEA04} and was also used in~\cite{BedrossianBlumenthalEA22,BlumenthalCotiZelatiEA22}.
We reproduce it here keeping track of the constants introduced in the proof and their dependence on $\kappa$ in order to prove that $\gamma$ in \eqref{e:X-exp-mix-disc-time} is $\kappa$-independent.

\begin{proof}[Proof of Lemma~\ref{l:X-exp-mix}]

 Let $\Z_0^d \defeq \Z^d - \set{0}$  and denote $\{e_m: m\in \Z^d\}$ as  the orthogonal basis $e_m(x) = e^{im\cdot x}$ for $L^2({\T^d})$ and denote
\begin{equation}
  f=\sum_{m\in \Z_0^d}\hat f_me_m \,,
  \quad
  g=\sum_{m\in \Z_0^d}\hat g_me_m
  \,,
\end{equation}
as the fourier expansions of $f$ and $g$. We note that  and $\hat f_0=\hat g_0=0$ as $f$ and $g$ are mean-zero.
\smallskip

Now, fix $\zeta>0$ and for $m, m' \in \Z_0^d$ and $\kappa>0$, define random variables
\begin{gather*}
N_{m,m'}^\kappa \defeq \max \set[\Big]{n \in \N ,
  ~\abs[\Big]{ \int e_m(x) e_{m'}\circ X^\kappa_{n}(x) \pi(dx) } > e^{-\zeta n}}\,, \\
K_\kappa \defeq \max\set[\Big]{\abs{m}\vee \abs{m'}: e^{\zeta N_{m,m'}^\kappa} > \abs{m}\abs{m'}}\,, \\
\widehat{D}_\kappa \defeq \max_{\abs{m}, \abs{m'}\leq K_\kappa} e^{\zeta N_{m, m'}^\kappa}\,.
\end{gather*}

Then by the definition of $N_{m,m'}^\kappa$ and Chebyshev, we get
\begin{align}
  \P[N_{m,m'}^\kappa > l] &\leq \sum_{n>l}\P\bigg[{~\abs[\Big]{ \int e_m(x) e_{m'}\circ X^\kappa_{n}(x) \pi(dx) } > e^{-\zeta n}}\bigg]\\
    &\leq \sum_{n > l} e^{2\zeta n} \E \abs[\Big]{ \int e_m(x) e_{m'}\circ X_n^\kappa(x) \pi(dx) }^2\,.
\end{align}
Observe
\begin{equation}
\E \abs[\Big]{ \int e_m(x) e_{m'}\circ X_n^\kappa(x) \pi(dx) }^2 = \int e_{m'}^{(2)}P^{(2),n}_\kappa e_m^{(2)}\,,
\end{equation}
where
\begin{gather}
e_m^{(2)}(x,y) \defeq e_m(x)\overline{e_m}(y)\,,\\
\pi^{(2)}(dx,dy) \defeq \pi(dx)\pi(dy) \,.
\end{gather}
From~\eqref{e: ugeGeneral}, we see
\begin{align}
  \abs[\Big]{ \int e_{m'}^{(2)}P^{(2),n}_\kappa e_m^{(2)} d\pi^{(2)}}
    &= \abs[\Big]{ \int \paren[\Big]{ e_{m'}^{(2)}P^{(2),n}_\kappa - \paren[\Big]{ \int e_{m'}^{(2)} d\pi^{(2)}}} e_m^{(2)} d\pi^{(2)}}
 \\
&\leq \int \abs[\Big]{e_{m'}^{(2)}P^{(2),n}_\kappa -\int e_{m'}^{(2)} d\pi^{(2)}} \, d\pi^{(2)}
\\
  &\leq Ce^{-\beta n}\|e_{m'}^{(2)}\|_V \int V d\pi^{(2)}
  = C_V e^{-\beta n}
  \,.
\end{align}
This implies
\begin{gather} \label{e: Nkk'}
\P[N_{m, m'}^\kappa > l] \leq C_V e^{(2\zeta-\beta) l}\,,
\end{gather}
provided
\begin{equation}
2\zeta-\beta < 0 \,. \label{e: zeta-beta}
\end{equation}
From now on, we make additional assumptions that $\zeta$ is small enough to satisfy
\begin{gather}
d+\frac{2\zeta-\beta}{\zeta}<0 \,, \label{e: d-zeta-beta}\\
\frac{1}{\zeta q} (2\zeta-\beta) + 1 < -1 \,, \label{e: q-zeta-beta}\\
\frac{5d}{2}+\frac{2\zeta-\beta}{\zeta} < -1 \label{e: d-zeta-beta2}\,.
\end{gather}
Equations~\eqref{e: Nkk'} and~\eqref{e: zeta-beta} imply that $\P( N_{m, m'}^\kappa < \infty ) = 1$ and we have the estimate
\begin{gather}
\abs[\Big]{ \int e_m(x) e_{m'}\circ X_n^\kappa(x) \pi(dx) } \leq e^{\zeta N_{m,m'}^\kappa - \zeta n}\,, \\
\text{ hence } \abs[\Big]{ \int f(x) g\circ X_n^\kappa(x) \pi(dx)} \leq e^{-\zeta n}\sum_{m, m'}\abs{\hat f_m}\abs{\hat g_{m'}}e^{\zeta N_{m,m'}^\kappa} \,. \label{e: corrFourier}
\end{gather}
Moreover, using~\eqref{e: Nkk'} and~\eqref{e: d-zeta-beta}, we observe that
\begin{align}
  \P[K_\kappa > l\big]
    &\leq 2 \sum_{m,m' \in \Z_0^d, \abs{m}>l} \P\bigg[e^{\zeta N_{m,m'}^\kappa} > \abs{m}\abs{m'}\bigg]
  \\
    &\leq 2 \sum_{m' \in \Z_0^2} \abs{m'}^\frac{2\zeta-\beta}{\zeta} \sum_{m \in \Z_0^2, \abs{m}>l} \abs{m}^\frac{2\zeta-\beta}{\zeta}
  \\
    &\lesssim \sum_{n=1}^\infty n^{\frac{2\zeta-\beta}{\zeta}+d-1} \sum_{n>l}^\infty n^{\frac{2\zeta-\beta}{\zeta}+d-1} \lesssim l^{d+\frac{2\zeta-\beta}{\zeta}}\,, \label{e: K tail bound}
\end{align}
where the constants in the inequalities are independent of $\kappa$.
Hence $\P( K_\kappa < \infty ) = 1$.

Noting that $$e^{\zeta N_{m,m'}^\kappa} \leq \widehat{D}_\kappa \abs{m}\abs{m'}$$ always holds, we conclude from~\eqref{e: corrFourier} that
\begin{gather}
	\abs[\Big]{\int f(x) g\circ X_n^\kappa(x) \pi(dx)} \leq  \widehat{D}_\kappa(\underline{\omega})e^{-\zeta n}\|f\|_{H^{\frac{d}{2}+2}} \|g\|_{H^{\frac{d}{2}+2}}\,.
\end{gather}
Finally, the same arguments in Lemma 7.1 and Section 7.3 of \cite{BedrossianBlumenthalEA22} show that for any $s, q>0$,
\begin{gather}
	\abs[\Big]{\int f(x) g\circ X_n^\kappa(x) \pi(dx)} \leq \widehat{D}_\kappa(\underline{\omega})e^{-(\frac{2s\zeta}{d+4}) n}\|f\|_{H^s} \|g\|_{H^s} \,.
\end{gather}
Moreover, using~\eqref{e: Nkk'} and~\eqref{e: K tail bound}, they show that
\begin{gather}
\E[\widehat{D}_\kappa^q] = \sum_{l=1}^\infty \E\big[1_{\{K_\kappa = l\}} \max_{\abs{m}, \abs{m'} \leq l} e^{\zeta q N_{m,m'}^\kappa} \big]\\
\leq \sum_{l=1}^\infty \P[K_\kappa = l]^\frac{1}{2} \norm[\Big]{\max_{\abs{m}, \abs{m'} \leq l} e^{\zeta q N_{m,m'}^\kappa}}_{L^2} \\
\lesssim \sum_{l=1}^\infty l^{d+\frac{2\zeta-\beta}{\zeta}}\bigg( \sum_{\abs{m}, \abs{m'} \leq l} \|e^{\zeta q N_{m,m'}^\kappa}\|_{L^2} \bigg)\\
\leq \paren[\Big]{1+\frac{\zeta q}{\beta - 2\zeta(1+q)}}^\frac{1}{2}\sum_{l=1}^\infty l^{\frac{5d}{2}+\frac{2\zeta-\beta}{\zeta}} < \infty \,,
\end{gather}
provided~\eqref{e: q-zeta-beta} and~\eqref{e: d-zeta-beta2}.
This completes the proof of the theorem with the choice of $\gamma\defeq \frac{2s\zeta}{d+4}$ which can be made independent of $\kappa$ since the conditions for $\zeta$, \eqref{e: zeta-beta}--\eqref{e: d-zeta-beta2}, are independent of $\kappa$.
\end{proof}

\section{An explicit Lyapunov function for Pierrehumbert flows.}\label{s:alt-shear-lyapunov}

It was proved in Section 5 of~\cite{BlumenthalCotiZelatiEA22} that the Pierrehumbert flows defined in Corollary~\ref{c:alt-shear} satisfy the assumptions~\ref{a:checkableConditions1}--\ref{a:checkableConditions4}. Then by Proposition~\ref{p:checkable-conditions}, we see that Assumption~\ref{a:fts}--\ref{a:submersion-2point} (and hence  Corollary~\ref{c:alt-shear}) must hold.
However, in the case of Pierrehumbert flows, we can explicitly construct a simple Lyapunov function and verify Assumption~\ref{a:Lyapunov} directly.

\begin{proposition}\label{p:alt-shear-lyapunov}
  Consider the RDS of alternating shears defined in Corollary~\ref{c:alt-shear}.
  If the flow amplitude~$A$ (in~\eqref{e:alt-shear-def}) is sufficiently large, then
  there exists~$s_*, p > 0$ such that the function~$V$ defined by
  \begin{equation}
    V(x, y) \defeq |x- y|_\infty^{-p} \text{ on }\Delta(s_*)\,,
  \end{equation}
  and extended continuously to~$\T^{d, (2)}$ is a Lyapunov function that satisfies Assumption~\ref{a:Lyapunov}.
  (Here $\abs{z}_\infty = \max_i \abs{z_i}$.)
\end{proposition}
\begin{proof}
In this proof, we use $C$ as a generic constant that doesn't depend on $A$ or $\kappa$.
For a given $x\in \T^2$, we write $x_1, x_2$ as the first and second coordinates of $x$, respectively.
Let $u(x, t)$ be defined as in Corollary~\ref{c:alt-shear}.

First, we note that for some small $s_*\in (0,\frac{1}{2})$ and any $x, y \in \Delta(s_*)$, we can find $\check x, \check y\in \R^2$ such that $[\check x]=x, [\check y]=y$, and
\begin{equation}
d(X_2^0(x), X_2^0(y))=\abs{\check X_2^0(\check x) - \check X_2^0(\check y)} \leq CA^2\abs{\check x - \check y} = CA^2 d(x,y) < \frac{1}{2}\,,
\end{equation}
so with slight abuse of notation, we still write $x, y, X_2^0(x),$ and $X_2^0(y)$ for $\check x, \check y, \check X_2^0(\check x),$ and $\check X_2^0(\check y)$, respectively.

We aim to show that there is a constant $0 < \beta < 1$ such that, for any $(x, y) \in \Delta(s_*)$, we have
\begin{equation}
  \label{eq:lyapunov-inequality-pierrehumbert}
  \E[V(\Phi_2(x), \Phi_2(y))] \leq \beta V(x, y),
\end{equation}
where $\Phi_t$ denotes the flow map induced by $u$ at time $t$.

Fix $(x, y) \in \Delta(s_*)$. First, note that we have $|x-y|_\infty \leq CA^2|\Phi_2(x)-\Phi_2(y)|_\infty$ for some large $A$ and therefore $V(\Phi_2(x), \Phi_2(y)) \leq {(CA^2)}^p V(x, y)$ almost surely. We break into two cases.
\restartcases
\case[$|x_2 - y_2| \geq |x_1 - y_1|$]
Let $E_0$ be the event
\begin{equation}
  E_0 = \set{ \abs{\Phi_1(x)_1 - \Phi_1(y)_1} < 2|x_2-y_2| }
  \,.
\end{equation}
Since $\abs{\Phi_2(x) - \Phi_2(y)}_\infty \ge \abs{\Phi_1(x)_1-\Phi_1(y)_1}_\infty$, we have
  \begin{align}
    \E[V(\Phi_2(x), \Phi_2(y))] &\leq \E[\abs{\Phi_1(x)_1-\Phi_1(y)_1}^{-p}]\\
                          &\leq \E[\abs{\Phi_1(x)_1-\Phi_1(y)_1}^{-p}\one_{E_0}] + \E[\abs{\Phi_1(x)_1-\Phi_1(y)_1}^{-p} \one_{E_0^c}]\\
    \label{e:EVphi2xPhi2y}
    &\leq {(CA^2)}^p V(x, y) P(E_0) + 2^{-p}V(x, y)\,.
  \end{align}

  We will now estimate $P(E_0)$.
  For this, we use the explicit form of the vector field $u$ to write
  \begin{equation*}
    \Phi_1(x)_1 - \Phi_1(y)_1 = x_1 - y_1 + A\sin(2\pi(x_2 - \zeta_0)) - A\sin(2\pi(y_2 - \zeta_0)).
  \end{equation*}
  Thus $E_0$ is contained in the event that
  \begin{equation*}
    |x_1-y_1| + |A\sin(2\pi(x_2 - \zeta_0)) - A\sin(2\pi(y_2 - \zeta_0))| < 2|x_2-y_2|.
  \end{equation*}
  In view of the assumption $\abs{x_2 - y_2} \geq \abs{x_1 - y_1}$, the above inequality is implied by
  \begin{equation}\label{eq:case1-sin-inequality}
    |A\sin(2\pi(x_2 - \zeta_0)) - A\sin(2\pi(y_2 - \zeta_0))| < |x_2-y_2|.
  \end{equation}
  Using the fundamental theorem of calculus, the left-hand side above can be written as the convolution $\abs{\left(2\pi A\cos(2\pi\cdot) \ast \mathbf{1}_{[x_2, y_2]}\right)(\zeta_0)}$.
  Here $[x_2, y_2]$ denotes the smallest interval (mod $\Z$) with $x_2$ and $y_2$ as endpoints.
  Since the derivative of $\cos(2\pi \cdot)$ is bounded away from zero near the zeros of $\cos(2\pi \cdot)$, the same is true for the convolution (rescaling by $|x_2-y_2|$).
  This implies that the set of $\zeta_0$ which satisfy~\eqref{eq:case1-sin-inequality} has measure at most $CA^{-1}$.
  This in turn implies~$\P(E_0) \leq C / A$.

  Using this in~\eqref{e:EVphi2xPhi2y} implies
  \begin{equation}
    \E[V(\Phi_2(x), \Phi_2(y))]
      \leq C A^{2p - 1} V(x, y)  + 2^{-p}V(x, y)\,.
  \end{equation}
  Choosing $p \in (0, \frac12)$ and $A > 0$ sufficiently large, we  can ensure
  \begin{equation}
    \E[V(\Phi_2(x), \Phi_2(y))]
      \leq 2^{-p / 2} V(x, y)
      \,,
  \end{equation}
  as desired.

\case[$|x_2 - y_2| < |x_1 - y_1|$]
Let $E_1$ be the event that $\abs{\Phi_1(x)_1-\Phi_1(y)_1} \leq A^{-1/2}\abs{x_1-y_1}$. Then $\P[E_1] \leq CA^{-1/2}$ by the same argument as in the previous case. On the other hand, if $E_2$ is the event that $\abs{\Phi_2(x)_2 - \Phi_2(y)_2} \leq 2A^{1/2}\abs{\Phi_1(x)_1-\Phi_1(y)_1}$, then we similarly have $\P[E_2] \leq CA^{-1/2}$. Putting these together, we conclude
  \begin{align*}
    \E[V(\Phi_2(x), \Phi_2(y))] &= \E[V(\Phi_2(x), \Phi_2(y)) \one_{E_1 \cup E_2}] + \E[V(\Phi_2(x), \Phi_2(y)) \one_{(E_1 \cup E_2)^c}]\\
    &\leq {(CA^2)}^{p}V(x, y)(CA^{-1/2}) + 2^{-p}V(x, y),
  \end{align*}
  so for $p \in (0, \frac14)$ we can choose $A > 0$ large enough to conclude as in the previous case.
\end{proof}

\bibliographystyle{halpha-abbrv}
\bibliography{gautam-refs1,gautam-refs2,preprints}
\end{document}